\pdfoutput=1
\documentclass[english]{jnsao}
\usepackage[utf8]{inputenc}
\usepackage[all]{xy}
\usepackage{tikz}
\usepackage{tikz-3dplot} 
\usepackage{pgfplots}
\pgfplotsset{compat=1.16}

\newcommand{\R}{\mathbb{R}}
\newcommand{\N}{\mathbb{N}}
\newcommand{\norm}[1]{\|#1\|}

\newcommand{\dist}[1]{{\mathrm{dist}}(#1)}

\newcommand{\mv}{\,\mid\,}
\newcommand{\B}{{\cal B}}

\newcommand{\K}{{\cal K}}

\newcommand{\Sp}{{\cal S}}
\newcommand{\F}{{\cal F}}

\newcommand{\skalp}[1]{\langle #1\rangle}

\newcommand{\xb}{\bar x}
\newcommand{\yb}{\bar y}
\newcommand{\zb}{\bar z}

\newcommand{\pb}{\bar p}

\newcommand{\Span}{{\mathrm{sp}\,}}
\newcommand{\cl}{{\mathrm{cl}\,}}

\newcommand{\inn}{{\mathrm{int}\,}}

\newcommand{\gph}{\mathrm{gph}\,}
\newcommand{\dom}{\mathrm{dom}\,}
\newcommand{\tto}{\rightrightarrows}

\tikzset{
	MyPersp/.style={scale=1.8,x={(-0.8cm,-0.4cm)},y={(0.8cm,-0.4cm)},
    z={(0cm,1cm)}},
	MyPoints/.style={fill=white,draw=black,thick}
		}

\usepackage[nameinlink,capitalize]{cleveref}

\title{On inner calmness*, generalized calculus, and derivatives of the normal cone mapping}

\author{Mat\'{u}\v{s} Benko\thanks{
	    Johannes Kepler University Linz,
	    Institute of Computational Mathematics,
        4040 Linz, Austria, and
        University of Vienna,
        Applied Mathematics and Optimization,
        1090 Vienna, Austria, \email{matus.benko@univie.ac.at}}}
\shortauthor{Benko}

\shorttitle{Inner calmness* and calculus}

\acknowledgements{
    The research was supported by the Austrian Science Fund (FWF) under grants P29190-N32 and P32832-N. The author would like to thank
Helmut Gfrerer and Ji\v{r}\'{i} V. Outrata for valuable discussions and support as well as two anonymous referees for their very helpful comments.}

\manuscriptlicense{CC-BY-SA 4.0}

\manuscriptsubmitted{2019-10-30}
\manuscriptaccepted{2021-06-23}
\manuscriptvolume{2}
\manuscriptnumber{5881}
\manuscriptyear{2021}
\manuscriptdoi{10.46298/jnsao-2021-5881}

\begin{document}

\maketitle

\begin{abstract}
 In this paper, we study continuity and Lipschitzian properties of
 set-valued mappings, focusing on inner-type conditions.
 We introduce new notions of inner calmness* and, its relaxation, fuzzy inner calmness*.
 We show that polyhedral maps enjoy inner calmness*
 and examine (fuzzy) inner calmness* of a multiplier mapping associated
 with constraint systems in depth.
 Then we utilize these notions to develop
 some new rules of generalized differential calculus,
 mainly for the primal objects (e.g. tangent cones).
 In particular, we propose an exact chain rule for graphical derivatives.
 We apply these results to compute the derivatives
 of the normal cone mapping, essential e.g.
 for sensitivity analysis of variational inequalities.
 \\[2ex]
	\\[2ex]
\end{abstract}


\section{Introduction}


The {\em normal cone mapping}, closely related to variational inequalities,
is a standard tool of variational analysis.
Among other things, it plays an important role in the formulation
of optimality conditions for constrained optimization problems.
Generalized derivatives of the normal cone mapping, in turn,
enable us to study various stability and sensitivity properties of
solutions of such optimization problems.
Moreover, they also offer a pathway to optimality conditions
for difficult problems, where the normal cone mapping
is used to describe the constraints (mathematical programs
with equilibrium constraints, bilevel programs, etc.).
For more information about the normal cone mapping
and related issues we refer to the standard textbooks \cite{DoRo14,Io17,Mo18,RoWe98}.

Recently, a lot of success has been achieved in the computation
of the {\em graphical derivative} of the normal cone mapping
$x \tto N_{\Gamma}(x)$,
where $\Gamma= \{x \mv g(x) \in D\}$ denotes
the feasible region of a constraint system.
Here, $g$ is continuously differentiable and $D$ is a closed convex set.
For the results with polyhedral $D$ see, e.g.,
\cite{BeGfrOut18,ChHi17,GfrOut14a,GfrOut18},
while the papers \cite{GfrMo17,HaMoSa17} extend the scope beyond polyhedrality.
Let us mention that most of these results actually deal with more
challenging cases when the feasible set depends on the parameter
or even when it depends on the parameter as well as on
the solution itself, see Section 5 for more details.

In general, the graphical derivative provides criteria for {\em isolated calmness}
and the {\em Aubin property}, see \cite[Corollary 4E.2 and Theorem 4B.2]{DoRo14}
for the precise statements and more details.
On top of that, it is also used to characterize {\em tilt stability}
of a local minimizer of a function, see \cite[Theorem 3.3]{ChHiNg17},
and, in the constrained case, the essential part
of this criterion is precisely
the graphical derivative of the normal cone mapping.
Interestingly, this criterion seems to be more applicable
than a similar tilt stability criterion based on the (limiting) coderivative;
see e.g. \cite{BeGfrMor18} for the application to second-order
cone programming.
The reason is that the standard calculus provides an upper
estimate for the limiting coderivative of the normal cone mapping,
but the estimate becomes exact only under additional assumptions,
see \cite{GfrOut16a}.
On the other hand, as suggested above, the graphical derivative
can be often fully computed under very mild assumptions.

The aforementioned results regarding the computation
of the graphical derivative
of the normal cone mapping were derived directly from the definition.
The main motivation for this paper is to better understand these computations.
More precisely, we ask the question,
whether there are some underlying calculus principles in play,
which could be identified and brought to light.

Generalized differential calculus belongs among the most fundamental
apparatus of modern variational analysis
\cite{AubinEkeland1984,AubinFrankowska2009,Io17,KlKum02,Mo18,Pen13,RoWe98}.
Arguably, more focus has been placed on the dual constructions
(normal cones, subdifferentials, coderivatives)
and this area is very developed and well-understood.
When it comes to the primal constructions
(tangent cones, subderivatives, graphical derivatives),
an obvious problem appears:
the lack of a suitable chain rule for graphical derivatives.
A closer look shows that the problem boils down to the following issue
regarding tangent cones.
Given a closed set $C$ and a continuously differentiable mapping $\varphi$,
let $\yb \in Q:=\varphi(C)$ and consider the estimate
\begin{equation}\label{eq:MainIssue}
    T_Q(\yb) \subset \bigcup_{(\xb \in C, \varphi(\xb) = \yb)}
    \nabla \varphi(\xb) T_C(\xb).
\end{equation}
It is well-known that the opposite inclusion always holds true
\cite[Theorem 6.43]{RoWe98}.
The question is whether there exists a reasonable (practically applicable)
assumption that guarantees this estimate.
Thus, let us now comment on the standard assumptions accompanying
the calculus rules.

Two main patterns can be observed throughout the most of the calculus formulas.
The first one can be represented by the image rule, where, as above,
the set $Q$ under investigation is generated as the (forward)
image of set $C$ under mapping $\varphi$.
The second one can be represented by the pre-image rule,
where the examined set $C$ is given as a (backward)
pre-image of another given set under a given map ($C:=\varphi^{-1}(Q)$).
We will refer to these patterns as {\em forward} and {\em backward}, respectively.

The main assumption utilized in the backward pattern,
typically called a qualification condition, is known to be
the calmness of the associated perturbation mapping,
often equivalently expressed via the metric subregularity of the feasibility mapping
\cite{HenJouOut02,HenOut05,IofOut08}.
We point out that calmness can be viewed as an
{\em outer} (also called upper) Lipschitzian property of set-valued mappings
since, roughly speaking, it resembles outer semicontinuity
with an additional Lipschitzian requirement regarding the rate of change of the map.
Both calmness and metric subregularity are central notions
of set-valued analysis with applications far beyond calculus.
Over the years, they drew attention of several renown researchers,
see the aforementioned publications as well as
\cite{DoRo04,FabHenKruOut10,GfrKl16,Iof79a,Rob81}
and the references therein.

The less-developed forward pattern seems to be linked with conditions of {\em inner-type}.
Indeed, the corresponding estimates for the dual objects
are known to hold under {\em inner semicompactness} and {\em inner semicontinuity},
see \cite{IofPen96,Mo06a}.
In order to obtain reasonable rules for the primal constructions,
it turns out that one can proceed by strengthening the above
continuity notions of inner-type
to suitable inner (also called lower) Lipschitzian notions,
such as {\em inner calmness}.

Interestingly, the term inner calmness was already coined
in \cite{BeGfrOut18a}, where we proposed calculus rules for
directional limiting normal cones and associated directional constructions.
It is not very surprising, however, since directional limiting normal cone is,
in a way, a primal-dual object:
it consists of limits of (dual) normals along (primal) tangent directions.
Moreover, inner calmness can be found in the literature
also under other names,
such as, e.g., Lipschitz lower semicontinuity \cite{KlKu15}
or recession with linear rate/linear recession \cite{CiFaKru18,Io17}.
In particular, \cite{CiFaKru18} contains a comprehensive
study of this property.
Let us also mention the stronger notion of
Lipschitz lower semicontinuity*,
recently introduced in \cite{CaGiHePa20}
and motivated by a relaxation of the Aubin property
from earlier works of Klatte \cite{Kl87,Kl94}.

This paper revolves around two new notions of {\em inner calmness*}
and, its relaxation, {\em fuzzy inner calmness*}.
Both these conditions are milder
than inner calmness, since the latter is based on inner
semicontinuity, while inner calmness* is based on (milder) inner semicompactness.
All the definitions as well as some basic properties
are provided in the preliminary Section 2.

The main part of the paper is Section 3, which contains
a justification of the new notions.
First, we show that polyhedral set-valued maps enjoy the inner calmness* property.
This provides an inner counterpart to the famous result on the calmness
(the upper Lipschitzness) of polyhedral maps by Robinson \cite{Rob81}.
More importantly, in Theorem \ref{the:Main_Inner_calmness},
we look into a general multiplier mapping associated
with constraint systems
and, under gradually strengthening assumptions, we show
its inner semicompactness, fuzzy inner calmness* and inner calmness*.

In Section 4, we easily derive several calculus rules of the forward pattern.
First, we prove the estimate \eqref{eq:MainIssue} under fuzzy inner calmness*
and then apply it to obtain a chain rule for graphical derivatives in exact form.
We also provide some useful corollaries of the chain rule,
such as the product rule in Theorem \ref{The : ProdRule}.

Finally, in Section 5, we return to the normal cone mapping
associated with constraint systems and apply our results to
compute or estimate its generalized derivatives.
In particular, the computation of the graphical derivative
is just a simple corollary of Theorems \ref{the:Main_Inner_calmness}
and \ref{The : ProdRule}.
For the most challenging setting,
where the feasible region $\Gamma(p,x)=\{z \in \mv g(p,x,z) \in D\}$
depends on the parameter $p$ as well as on the
decision variable $x$
($x$ typically corresponds to the solution of an underlying
parametrized variational inequality),
we recover the state of the art results from \cite{BeGfrOut18}.
This suggests that our calculus based on
(fuzzy) inner calmness* is applicable in quite relevant situations.
Using the analogous calculus based on inner calmness is also possible
if one significantly strengthens the assumptions imposed on the constraints
(nevertheless, one still recovers the recent results from \cite{GfrOut18}).
Using the standard result from \cite[Theorem 6.43]{RoWe98}, however,
seems completely out of question, since it requires convexity of the graph
of the aforementioned multiplier mapping.

As a specific application, we discuss the {\em semismoothness*} of the normal cone mapping,
based on the estimates of its directional limiting coderivative.
Semismoothness* is an extension of the standard semismoothness property
to set-valued mappings. It was recently introduced in
\cite{GfrOut19} and used to design a novel Newton method for generalized equations.
The notion was then further explored e.g. in \cite{FrGoHo21,KhMoVo20}.

The following notation is employed. Given a set $A \subset \R^{n}$,
the closure and interior of $A$ are denoted, respectively,
by $\cl A$ and $\inn A$,
$\Span A$ stands for the linear hull of $A$ and $A^{\circ}$
is the (negative) polar cone of $A$.
A set is called cone if any nonnegative multiple of its element
also belongs to the set.
If $A=\{a\}$, we use just $[a]$ instead of $\Span \{a\}$
and $[a]^{\perp}$ stands for its orthogonal complement.
We denote by
$\dist{\cdot,A}:=\inf_{y\in A}\norm{\cdot-y}$
the usual point to set distance with the convention
$\dist{\cdot,\emptyset}=\infty$ and
$\skalp{\cdot,\cdot}$ stands for the standard scalar product.
Further, $\B$, $\Sp$ stand respectively for the closed unit ball
and the unit sphere of the space in question.
Given a (sufficiently) smooth function $f:\R^n\to\R$,
denote its gradient and Hessian at $x$ by $\nabla f(x)$ and $\nabla^2 f(x)$, respectively.
Considering further a vector function $\varphi :\R^n\to\R^s$ with $s>1$,
denote by $\nabla \varphi(x)$ the Jacobian of $\varphi$ at $x$,
i.e., the mapping $x \to \nabla \varphi(x)$ goes from
$\R^n$ into the space of $s\times n$ matrices, denoted by $(\R^s)^n$.
Moreover, for a mapping $\beta: \R^n \to (\R^s)^m$ and a
vector $y \in \R^s$, we introduce
the scalarized map $\skalp{y,\beta}:\R^n\to \R^m$ given by
$\skalp{y,\beta}(x)=\beta(x)^T y$.
Following traditional patterns, given a number $\alpha \geq 0$
we denote by $o(t^{\alpha})$ for $t \geq 0$ a term with the
property that $o(t^{\alpha})/t^{\alpha} \to 0$ when $t \downarrow 0$.
Particularly, $o(1)$ stands for a term which converges to $0$ as
the variable in question goes to $0$.
Finally, if no confusion arises, we denote a sequence $(x_k)_{k=1}^{\infty} \subset \R^n$ simply by $(x_k)$ and very often we also drop the brackets and write just $x_k$.


\section{Preliminaries}


We begin by recalling several definitions and results from variational analysis.
Let $\Omega\subset\R^n$ be an arbitrary closed set and $\xb\in\Omega$.
The {\em tangent} (also called  {\em Bouligand} or {\em contingent}) {\em cone} to $\Omega$ at $\xb$ is given by
\[T_{\Omega}(\xb):=\{u\in \R^n\mv \exists (t_k)\downarrow 0, (u_k)\to u: \xb+t_ku_k\in\Omega \ \forall \, k\}.\]
We denote by
\begin{equation*}
\widehat N_{\Omega}(\xb):=T_{\Omega}(\xb)^\circ
=\{x^* \in \R^n \mv \skalp{x^*,u} \leq 0 \ \forall \, u \in T_{\Omega}(\xb)\}
\end{equation*}
the {\em regular} ({\em Fr\'echet}) {\em normal cone} to $\Omega$ at $\xb$.
The {\em limiting} ({\em Mordukhovich}) {\em normal cone} to $\Omega$ at $\xb$
is defined by
\[N_{\Omega}(\xb):=\{x^* \in\R^n \mv \exists (x_k)\to \xb,\ (x_k^*)\to x^*: x_k^*\in \widehat N_{\Omega}(x_k) \ \forall \, k\}.\]
Finally, given a direction $u\in \R^n$, we denote by
\[N_\Omega(\xb;u):=\{x^* \in\R^n \mv \exists (t_k)\downarrow 0, (u_k)\to u, (x_k^*)\to x^*: x_k^* \in \widehat N_{\Omega}(\xb+t_ku_k) \ \forall \, k\}\]
the {\em directional limiting normal cone} to $\Omega$ at $\xb$
{\em in direction} $u$.

If $\xb \notin \Omega$, we put $T_{\Omega}(\xb)=\emptyset$, $\widehat N_{\Omega}(\xb)=\emptyset$, $N_{\Omega}(\xb)=\emptyset$ and $N_\Omega(\xb;u)=\emptyset$. Further note that
$N_\Omega(\xb;u)=\emptyset$ whenever $u\not\in T_\Omega(\xb)$.
If $\Omega$ is convex, then $\widehat N_\Omega(\xb)= N_\Omega(\xb)$ amounts to the classical normal cone in the sense of convex analysis and we will write $N_\Omega(\xb)$. More generally, we say that $\Omega$ is
{\em Clarke regular} at point $\xb$, provided $\widehat N_\Omega(\xb)= N_\Omega(\xb)$ and we write $N_\Omega(\xb)$ in such case.

The following generalized derivatives of set-valued mappings are defined by means of tangents and normals to the graph of the mapping.
Let $M:\R^n\tto\R^m$ be a set-valued map with closed graph
$\gph M := \{(x,y) \mv y \in M(x)\}$ and let $(\xb,\yb)\in\gph M$.
The mapping $D M(\xb,\yb):\R^n\tto\R^m$, defined by
\[
DM(\xb,\yb)(u):= \{v\in\R^m\mv (u,v)\in T_{\gph M}(\xb,\yb)\},
\]
is called the {\em graphical derivative} of $M$ at $(\xb,\yb)$.
 The mapping $\widehat D^\ast M(\xb,\yb):\R^m\tto\R^n$
 \[\widehat D^\ast M(\xb,\yb)(v^\ast):=\{u^\ast\in \R^n \mv (u^\ast,- v^\ast)\in \widehat N_{\gph M}(\xb,\yb )\}\]
is called the {\em regular (Fr\'echet) coderivative} of $M$ at $(\xb,\yb)$.
 The mapping $D^\ast M(\xb,\yb): \R^m\tto\R^n$
\[ D^\ast M(\xb,\yb)(v^\ast):=\{u^\ast\in \R^n \mv (u^\ast,- v^\ast)\in N_{\gph M}(\xb,\yb )\}\]
is called the {\em limiting (Mordukhovich) coderivative} of $M$ at $(\xb,\yb)$.
 Given a pair of directions $(u,v)\in\R^n\times\R^m$, the
 mapping $D^\ast M((\xb,\yb ); (u,v)):\R^m\tto\R^n$, given by
\begin{equation*}
D^\ast  M((\xb,\yb ); (u,v))(v^\ast):=\{u^\ast \in \R^n \mv (u^\ast,-v^\ast)\in N_{\gph M}((\xb,\yb ); (u,v)) \},
\end{equation*}
is called the {\em directional limiting coderivative} of $M$ at $(\xb,\yb)$ {\em in direction} $(u,v)$.

In the rest of the section, we recall some well-known continuity
and Lipschitzian notions for set-valued maps and introduce some new ones.
As advertised in Introduction, we will mainly focus
on the inner-type properties.
In order to define also directional versions of these notions,
we employ the following terminology.
Given $x, u \in \R^n$, we say that a sequence $x_k$ converges to $x$ {\em from direction} $u$ if there exist
$t_k \downarrow 0$ and $u_k\to u$ with $x_k=x+t_ku_k$.

Recall that $S: \R^m \rightrightarrows \R^n$ is {\em inner semicompact at $\yb$ with respect to (wrt)} $\Omega \subset \R^m$ if for every sequence $y_k \to \yb$ with $y_k \in \Omega$
there exists a subsequence $K$ of $\N$ and a sequence $(x_k)_{k\in K}$ with $x_k \in S(y_k)$ for $k \in K$
converging to some $\xb$.
Given $\xb \in S(\yb)$, we say that $S$ is {\em inner semicontinuous at $(\yb,\xb)$ wrt} $\Omega$
if for every sequence $y_k \to \yb$ with $y_k \in \Omega$ there exists
a sequence $x_k \to \xb$ with $x_k \in S(y_k)$ for
sufficiently large $k$.
If $\Omega=\R^m$, we speak only about inner semicompactness at $\yb$ and inner semicontinuity at $(\yb,\xb)$.
If we restrict ourselves to sequences $y_k$ converging to $\yb$
from a fixed direction $v \in \R^m$, we speak of
inner semicompactness and inner semicontinuity
{\em in direction} $v$.
For more details regarding these standard notions we refer to \cite{Mo06a}.

We point out that inner semicompactness is implied by the
simpler, more intuitive, local boundedness condition \cite[Definition 5.14]{RoWe98}. In fact, local boundedness is often imposed in the development
of calculus instead of inner semicompactness,
see \cite{RoWe98}.
We believe, however, that the boundedness assumption
is slightly misleading, since, as the name suggests,
its purpose is to {\em restrict} the mapping,
while inner semicompactness, in a sense, says the opposite
- be as unbounded as you like, I just need a convergent subsequence.
In Section 3, we will see the impact of these differences
when dealing with a multiplier map.

For the purposes of this paper, it suffices
to say that $S$ is {\em outer semicontinuous} (osc)
if $\gph S$ is closed, see \cite[Theorem 5.7]{RoWe98}.
Hence, when dealing with derivatives of set-valued maps,
given as normals or tangents to their graphs,
we will use osc as a standing assumption,
since we want the graphs to be closed.

In Section 4, we will be interested in tangents and normals to sets,
which are actually domains of certain maps.
The following lemma shows that inner semicompactness
of an osc map guarantees local closedness of its domain.
Clearly, local closedness, defined below, is sufficient for our needs.
See \cite[Theorem 5.25 (b)]{RoWe98} for a similar
result based on local boundedness.

\begin{lemma}\label{Lem : iscomp_dom_closed}
Let $S: \R^m \rightrightarrows \R^n$ be osc
and inner semicompact at $\yb$ wrt
$\dom S := \{y \in \R^m \mv S(y) \neq \emptyset\}$.
Then, $\dom S$ is locally closed around $\yb$,
i.e., $\dom S \cap V$ is closed for some
closed neighbourhood of $\yb$.
\end{lemma}
\begin{proof}
 By contraposition, assume that for every $k$
 the set $\dom S \cap (\yb + (1/k)\B)$ is not closed,
 i.e., there exists $y_k \notin \dom S$
 with $\norm{y_k - \yb} \leq 1/k$, together with a sequence
 $(y_k^l)_{l = 1}^{\infty} \subset \dom S$
 such that $y_k^l \to y_k$ as $l \to \infty$
 and $\norm{y_k^l - \yb} \leq 1/k$ for all $l \in \N$.
 
 Now if there exists $k_0$ such that for all $l$
 there exists some $x_{k_0}^l \in S(y_{k_0}^{l})$
 with $\norm{x_{k_0}^l} \leq k_0$, then
 $x_{k_0}^l$ must converge to some $x_{k_0}$ along a subsequence. Since $(y_{k_0}^l,x_{k_0}^l) \in \gph S$
 and $\gph S$ is closed, we infer $(y_{k_0},x_{k_0}) \in \gph S$, which contradicts $y_{k_0} \notin \dom S$.
 This means that for each $k$ there exists $l_k$ such that
 $x_k^{l_k} \in S(y_k^{l_k})$ implies $\norm{x_k^{l_k}} > k$.
 Then, however, we get $\dom S \ni y_k^{l_k} \to \yb$
 as $k \to \infty$ and the inner semicompactness of $S$
 yields the existence of $x_k^{l_k} \in S(y_k^{l_k})$
 converging to some $\xb$ along a subsequence,
 contradicting $\norm{x_k^{l_k}} > k$.
 This completes the proof.
\end{proof}

In \cite{BeGfrOut18a}, we needed to strengthen the
above continuity properties by controlling
the rate of convergence $x_k \to \xb$.
To this end, we came up with {\em inner calmness}
based on inner semicontinuity.
Here we also introduce the following milder concept
of inner calmness* based on inner semicompactness.

\begin{definition}
 A set-valued mapping $S: \R^m \rightrightarrows \R^n$ is called
 \begin{itemize}
     \item[(i)] {\em inner calm* at $\yb \in \R^m$ wrt} $\Omega \subset \R^m$
  if there exists $\kappa > 0$ such that for every sequence $y_k \to \yb$ with $y_k \in \Omega$, there exist a subsequence $K$ of $\N$, together with a sequence $(x_k)_{k\in K}$ and
  $\xb \in \R^n$ with $x_k \in S(y_k)$ for $k \in K$ and
  \begin{equation}\label{eq:ICdef}
   \norm{x_k - \xb} \leq \kappa \norm{y_k - \yb};
  \end{equation}
     \item[(ii)] {\em inner calm at $(\yb,\xb) \in \gph S$ wrt} $\Omega$
  if there exist $\kappa > 0$ such that for every
  sequence $y_k \to \yb$ with $y_k \in \Omega$
  there exists a sequence $x_k$ satisfying
  $x_k \in S(y_k)$ and \eqref{eq:ICdef} for sufficiently large $k$.
 \end{itemize}
As before, if we restrict ourselves to sequences $y_k$ converging to $\yb$
from a fixed direction $v \in \R^m$, we speak of
inner calmness* and inner calmness
{\em in direction} $v$.
\end{definition}

\noindent
Clearly, we have the implications:
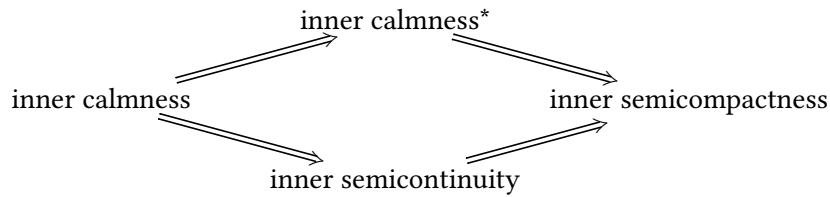
\[
  \begin{xy}\xymatrix@C=110pt@R=30pt@!0{
   & \text{inner calmness*} \ar@{=>}[dr] & \\
   \text{inner calmness} \ar@{=>}[dr]\ar@{=>}[ur] & & 
   \text{inner semicompactness} \\
   & \text{inner semicontinuity}\ar@{=>}[ur] &
}
\end{xy}
\]

Note that each of the four inner conditions implies that $S(y) \neq \emptyset$ for $y \in \Omega$ near $\yb$. While this can be desirable in some situations,
it can also be quite restrictive.
For our purposes, however, we will often consider these properties wrt to the domain of $S$,
adding no restriction at all.

We also consider one important outer Lipschitzian notion.
We say that $S$ is {\em calm} at $(\yb,\xb) \in \gph S$,
provided there exists $\kappa > 0$ such that for every sequence $x_k \to \xb$
for which there exists a sequence $y_k$ with $x_k \in S(y_k)$,
there exists a sequence $\tilde x_k$ satisfying $\tilde x_k \in S(\yb)$ and
\begin{equation} \label{eq : CalmnessDef_seq}
\norm{x_k - \tilde x_k} \leq \kappa \norm{y_k - \yb}
\end{equation}
for sufficiently large $k$.
Interestingly, in the definition of calmness, the crucial
sequence is $x_k$ in the image space $\R^n$.
By the same token, $S$ is called calm {\em in direction}
$u \in \R^n$ if the above holds for all sequences $x_k$
converging to $\xb$ from $u$.
This is also related to the well-known fact that calmness of $S$ at $(\yb,\xb)$
is equivalent to {\em metric subregularity} of $M:=S^{-1}$ at $(\xb,\yb)$.

Finally, we employ the following terminology.
We say that $S$ is calm at $(\yb,\xb)$ with {\em constant}
$\kappa \geq 0$ if $\kappa$ satisfies \eqref{eq : CalmnessDef_seq}.
The infimum of all calmness constants is called
the {\em calmness modulus} of $S$,
which is set to be $+\infty$ if $S$ fails to be calm.
Naturally, the same applies to inner calmness* and
inner calmness as well as to the directional versions
of these properties.

\subsection{Elementary results and other notions}

Calmness and inner calmness are typically defined
via neighbourhoods instead of sequences.
First, we show that these definitions coincide.

\begin{lemma}
 Let $S: \R^m \rightrightarrows \R^n$ and $(\yb,\xb) \in \gph S$.
 Then
 \begin{itemize}
     \item[(i)] $S$ is inner calm at $(\yb,\xb)$ wrt $\Omega$
     with modulus $\bar\kappa$ if and only if
     $\bar\kappa$ is the infimum of $\kappa$ over all
     combinations of $\kappa$
     and neighbourhoods $V$ of $\yb$ satisfying
     \begin{equation} \label{eq : InnerCalmnessDef}
     \xb \in S(y) + \kappa \norm{y - \yb}\B \quad \forall \, y \in V \cap \Omega;
     \end{equation}
     \item[(ii)] $S$ is calm at $(\yb,\xb)$
     with modulus $\bar\kappa$ if and only if
     $\bar\kappa$ is the infimum of $\kappa$ over all
     combinations of $\kappa$
     and neighbourhoods $U$ of $\xb$ satisfying
     \begin{equation} \label{eq : CalmnessDef}
      S(y) \cap U \subset S(\yb) + \kappa \norm{y - \yb}\B \quad 
      \forall \, y \in \R^m.
     \end{equation}
 \end{itemize}
\end{lemma}
\begin{proof}
 Clearly, if $\kappa$ satisfies \eqref{eq : InnerCalmnessDef}
 for some neighbourhood $V$ of $\yb$, then it also satisfies
 \eqref{eq:ICdef} from the definition of inner calmness.
 On the other hand, if $\kappa$ satisfies \eqref{eq:ICdef},
 then there exists a neighbourhood $V$ of $\yb$ such that
 \eqref{eq : InnerCalmnessDef} holds.
 Indeed, if not, we find a sequence $\Omega \ni y_k \to \yb$ with
 \[
   \dist{\xb,S(y_k)} > \kappa \norm{y_k - \yb},
 \]
 violating \eqref{eq:ICdef}. Hence, the infimum defining
 modulus equals the infimum from the statement.
 
 The proof for calmness is analogous.
\end{proof}
\noindent
Note that calmness is sometimes defined by
\eqref{eq : CalmnessDef} with $\R^m$ replaced
by a neighbourhood of $\yb$.
As \cite[Exercise 3H.4]{DoRo14} clarifies,
these definitions are in fact equivalent.
Clearly, calmness and inner calmness are implied by the so-called
Aubin property \cite[Definition 9.36]{RoWe98}.

For the sake of completeness, we write down the
neighbourhood-based definition of metric subregularity.
\begin{definition}\label{Def: MS}
  Let $M: \R^n \rightrightarrows \R^m$ and $(\xb,\yb) \in \gph M$. We say that
  $M$ is {\em metrically subregular} at $(\xb,\yb)$ provided there exist $\kappa > 0$ and
  a neighbourhood $U$ of $\xb$ such that
  \[\dist{x,M^{-1}(\yb)} \leq \kappa \dist{\yb,M(x)} \quad \forall \, x \in U.\]
\end{definition}

Next we look into the relation between the standard and the directional
versions of the various calmness properties.
We note that our interest in the directional approach
stems from its successful implementation
and development in recent years by Gfrerer \cite{Gfr13a}.

\begin{lemma}\label{Lem:DirVsStand}
 Given $S: \R^m \rightrightarrows \R^n$, $\yb \in \R^m$
 and $\Omega \subset \R^m$, let $\bar\kappa$ denote
 the modulus of inner calmness* of $S$ at $\yb$ wrt $\Omega$
 (inner calmness of $S$ at $(\yb,\xb) \in \gph S$ wrt $\Omega$\,,
 calmness of $S$ at $(\yb,\xb)$) and let $\bar\kappa_v$
 denote the corresponding modulus in direction $v \in \R^m$
 (for calmness $\bar\kappa_u$ with $u \in \R^n$).
 Then
 \begin{equation}\label{eq:StandVSdirMod}
     \bar\kappa = \max_{v \in \Sp}\bar\kappa_v \quad
     (\bar\kappa = \max_{u \in \Sp}\bar\kappa_u \ \textrm{ for calmness}).
 \end{equation}
\end{lemma}
\begin{proof}
 Let $\kappa$ by any constant of inner calmness* of $S$.
 Then, clearly, for any direction $v$ we have $\bar\kappa_v \leq \kappa$ and hence $\bar\kappa_v \leq \bar\kappa$ by the definition
 of infimum and $\sup_{v \in \Sp} \bar\kappa_v \leq \bar\kappa$ follows.
 Naturally, if $\bar\kappa = \infty$, the inequality holds as well.
 
 On the other hand, let $\kappa < \bar \kappa$.
 Then there exists a sequence $y_k \in \Omega$ converging to $\yb$
 such that for any $\xb \in \R^n$ there exists $k_0$ such that
 for all $x_k \in S(y_k)$ with $k \geq k_0$ one has
 \[\norm{x_k - \xb} > \kappa \norm{y_k - \yb}.\]
 By passing to a subsequence, however, we may assume that
 there exists $v \in \Sp$ such that
 $(y_k - \yb)/\norm{y_k - \yb} \to v$
 and hence we conclude that $\kappa$ can not be an inner calmness*
 constant of $S$ in direction $v$, i.e., $\kappa \leq \bar\kappa_v$.
 Thus $\bar\kappa \leq \bar\kappa_v$ holds as well
 and \eqref{eq:StandVSdirMod} follows.

 The proofs for inner calmness and calmness follow by the same steps.
\end{proof}
\noindent
We point out that \eqref{eq:StandVSdirMod} yields, in particular,
that the standard (nondirectional) version of any of the calmness properties
is equivalent to the validity of that property in every direction
from the unit sphere. Indeed, $\bar\kappa = +\infty$ if and only if
there exists a direction $v$ with $\bar\kappa_v = +\infty$.

We will also use the following relaxed version of inner calmness*.

\begin{definition}
A set-valued mapping $S: \R^m \rightrightarrows \R^n$ is called
{\em inner calm* at $\yb \in \R^m$ in direction $v \in \R^m$ wrt} $\Omega \subset \R^m$ {\em in the fuzzy sense},
  if, either $v \notin T_{\Omega}(\yb)$,
  or there exist $\kappa_v > 0$ ({\em constant} of fuzzy inner calmness* in direction $v$) together with sequences
  $y_k \in \Omega$ converging to $\yb$ from $v$
  and $x_k \in S(y_k)$ and a point $\xb \in \R^n$ such that
  \begin{equation*}
   \norm{x_k - \xb} \leq \kappa_v \norm{y_k - \yb}.
  \end{equation*}
  The infimum $\bar\kappa_v$ of all such $\kappa_v$ is called
  the {\em modulus} of fuzzy inner calmness* in direction $v$
  with the convention that $\bar\kappa_v := 0$ if $v \notin T_{\Omega}(\yb)$.
  We say that $S$ is {\em inner calm* at $\yb$ wrt}
  $\Omega \subset \R^m$ {\em in the fuzzy sense},
  provided it is inner calm* in the fuzzy sense
  in every direction $v \in \Sp$.
\end{definition}
\noindent
The case $v \notin T_{\Omega}(\yb)$ is included because
if there is no sequence $y_k \in \Omega$ converging to $\yb$ from $v$,
there is no requirement to meet.
Moreover, it is also needed in order to make sure that
inner calmness* in the fuzzy sense is implied by inner calmness*.

Note also that we do not define the (nondirectional) modulus of fuzzy inner calmness*. The reason is that, unlike in the case of the other
calmness properties, it may happen that a map is inner calm*
in the fuzzy sense in every direction, yet the supremum
of the moduli over the unit directions blows up.

\begin{example}\label{Example_1}
Consider the set $G \subset \R^3$ given as a curve
$t \to (y_1(t),y_2(t),x(t))$ for $t \in (0,2\pi]$ with
\begin{equation} \label{eq:CurveDef}
y_1(t) = \cos(t), \ y_2(t) = \sin(t), \ x(t) = \frac{1}{t} - \frac{1}{2\pi},
\end{equation}
see Figure 1.
\begin{figure}
\begin{center}
\begin{tikzpicture}
\begin{axis}[smooth,view={110}{30}]
\addplot3+[very thick,mark=none,domain=0.1:2*pi,samples=200,samples y=0]
	({cos(deg(x))},
	 {sin(deg(x))},
	 {1/x - 1/(2*pi)});
\end{axis}
\end{tikzpicture}
\end{center}
\caption{Sketch of set $G$.}
\end{figure}
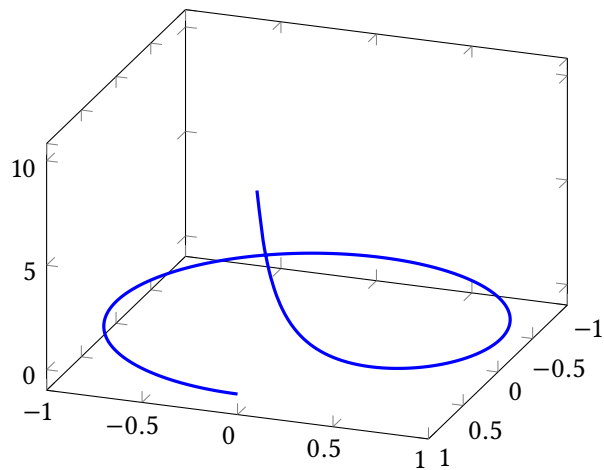
Now let $S:\R^2 \tto \R$ be the set-valued map whose graph consists
of rays starting from the origin $(0,0,0)$
and passing through the points of $G$, i.e.,
$\gph S= \{r G \mv r \geq 0\}$.
Consider point $\yb=(0,0)$. It is easy to see
that for any unit direction $(\cos(t), \sin(t))$ with $t \in (0,2\pi)$,
the modulus of inner calmness* and fuzzy inner calmness* coincide
and equal $1/t - 1/(2\pi)$.
Clearly, this quantity goes to infinity as $t$ goes to $0$,
while for $t=2\pi$ it becomes $0$.
In fact, this precisely entails the issue with direction $v=(1,0)$.
Indeed, on one hand it means that $S$ is inner calm* at $\yb$
in the fuzzy sense (the modulus in direction $v$ is $0$).
On the other hand, from Lemma \ref{Lem:DirVsStand} we conclude
that for the modulus of inner calmness* we have
$\bar\kappa=\bar\kappa_v=+\infty$ and hence $S$ is not inner calm* at $\yb$.
\end{example}

Finally, in case of a single-valued mapping $\varphi: \R^m \to \R^n$,
calmness at $\yb$ is defined as the existence of $\kappa>0$
and a neighbourhood $V$ of $\yb$ such that
\begin{equation} \label{eq : CalmnessDefSingValued}
  \norm{ \varphi(y) - \varphi(\yb)} \leq \kappa \norm{y - \yb} \quad \forall \, y \in V,
\end{equation}
or, equivalently, the existence of $\kappa>0$ such that
\eqref{eq : CalmnessDefSingValued}, with $y$ replaced
by iterates $y_k$ of an arbitrary sequence $y_k \to \yb$,
holds for sufficiently large $k$.
If it holds for $y_k$ converging to $\yb$ from
a direction $v$, we say $\varphi$ is calm at $\yb$
in direction $v$.
Interestingly, the above definition
coincides with the definition of inner calmness of $\varphi$
at $(\yb,\varphi(\yb))$, but {\em not}
with calmness (due to neighbourhood $U$)
or inner calmness* (due to not requiring $\xb \in S(\yb)$).
Naturally, if we restrict ourselves to continuous mappings,
all three notions coincide.
Further, $\varphi$ is called {\em Lipschitz continuous near} $\yb$ if the inequality
\[\norm{\varphi(y) - \varphi(y^{\prime})} \leq \kappa \norm{y - y^{\prime}} \ \forall \, y, y^{\prime} \in V\]
is fulfilled with $\kappa > 0$ and $V$ being a neighbourhood of $\yb$.

%
%
%
%
%


\section{Inner semicompactness and inner calmness*}

The role of inner calmness* as an assumption is highlighted
in the next section dealing with the calculus rules.
In this section, we discuss two interesting cases when it is satisfied.
In both cases, we first prove some basic result in terms of inner semicompactness,
which gets improved to inner calmness* (in the fuzzy sense)
after we add suitable polyhedrality assumptions.

\subsection{Polyhedral set-valued maps}

We begin by the simple example showing the limitations
of inner semicontinuity and inner calmness,
namely, that the lack of convexity of the graph can easily
lead to violation of these properties.
\begin{example}\label{ex:simple}
 Let $S:\R \tto \R$ by given by
 \begin{equation*}
 S(y)= \left\{\begin{array}{rl}
              0 & \textrm{ for } y \leq 0,\\
              1 & \textrm{ for } y \geq 0.
             \end{array}\right.
\end{equation*}
It is easy to see that $S$ is not inner semicontinuous at $(\yb,\xb) = (0,0)$,
due to $y_k := 1/k \to \yb$ and $S(y_k) = 1 \not\to 0$, or at $(\yb,\xb) = (0,1)$,
due to $y_k := -1/k \to \yb$ and $S(y_k) = 0 \not\to 1$.

On the other hand, $S$ is clearly inner semicompact (even inner calm*) at $\yb$.
Indeed, given a sequence $y_k \to \yb$, we can {\em choose} $\xb$ to be
either $0$ or $1$, depending on which of the two values is attained by $S(y_k)$ infinitely many times.
\end{example}

Next, consider a map $S : \R^m \tto \R^n$
whose graph is a finite union of sets $G_i$ for $i = 1, \ldots,l$
and denote by $S_i$ the maps with $\gph S_i = G_i$,
referred to as the components of $S$.
If $G_i$ are closed, i.e., components $S_i$ are osc,
then so is $\gph S$ and hence $S$ is also osc.
We will now show that the properties of inner semicompactness
and inner calmness* are also preserved under finite unions.

\begin{lemma}
 Given a map $S : \R^m \tto \R^n$, assume that its components $S_i$
 are inner semicompact (inner calm*) at $\yb \in \R^m$ wrt to $\dom S_i$.
 Then $S$ is inner semicompact (inner calm*) at $\yb$ wrt to its domain.
\end{lemma}
\begin{proof}
 Let $y_k \to \yb$ with $y_k \in \dom S$. By passing to a subsequence if necessary,
 we may assume that there exists $i$ with $y_k \in \dom S_i$ for all $k$.
 The inner semicompactness (inner calmness*) of $S_i$
 yields the existence of a sequence  $x_k \in S_i(y_k)$ converging to some $\xb$
 (and $\kappa_i \geq 0$ with $\norm{x_k - \xb} \leq \kappa_i \norm{y_k - \yb}$).
 Since $x_k \in S(y_k)$, the claim follows
 (with $\kappa := \max_{i = 1, \ldots,l} \kappa_i$).
\end{proof}

The question is how to apply this lemma.
It is known that mappings with convex graphs
are inner semicontinuous at any $(\yb,\xb)$ with $\yb$
in the interior of the domain \cite[Theorem 5.9(b)]{RoWe98}.
The following example shows, however, that even a map with
closed convex graph and domain may fail to be inner semicompact
wrt its domain at every point of the boundary of domain.
\begin{example} \label{Ex_isc_violated}
Consider again the set $G \subset \R^3$
from Example \ref{Example_1}.
Now let $S:\R^2 \tto \R$ be the set-valued map whose graph is
the closure of the convex hull of $G$. Note also that the domain of $S$
is the closed unit ball. We claim that $S$ is not inner semicompact at $\yb=(1,0)$ wrt to $\dom S$.
Indeed, consider a sequence $t_k \downarrow 0$
and set $y_k := (\cos(t_k),\sin(t_k)) \to \yb$.
We will show that there is no $x_k \in S(y_k)$ with $x_k < 1/t_k \to \infty$.
To this end, for every $k$ we construct a halfspace containing set $G$
but no $(y_k,x_k)$ with $x_k < 1/t_k$.
Set $q_k := t_k^2 (1/t_k - 1/(2\pi))$,
\[b_k := \left( \sin(t_k) + q_k \cos(t_k), 1 - \cos(t_k) + q_k \sin(t_k),
-t_k^2(1 - \cos(t_k))\right) \]
and consider the halfspace
\[\mathcal{H}_k:=\{z \in \R^3 \mv \skalp{b_k,z} \leq \sin(t_k) + q_k \cos(t_k)\}.\]

Clearly, $(1,0,0) \in \mathcal{H}_k$.
Moreover, let us explain that $G \subset \mathcal{H}_k$,
i.e., that every point $(y,x)$
of the form \eqref{eq:CurveDef} also belongs to $\mathcal{H}_k$.
Consider the function
\[h_k(t) := \skalp{b_k,(\cos(t),\sin(t),1/t - 1/(2\pi))}.\]
A simple computation yields
$h_k(t_k) = \sin(t_k) + q_k \cos(t_k)$.
Thus, it suffices to show that $h_k$
attains its global maximum over $(0,2\pi]$ at $t_k$.
Since
\[h_k^{\prime}(t) = \skalp{b_k,(-\sin(t),\cos(t),-1/t^2)} \quad \textrm{and} \quad
h_k^{\prime \prime}(t) = \skalp{b_k,(-\cos(t),-\sin(t),2/t^3)},\]
we get $h_k^{\prime}(t_k)=0$ and
$h_k^{\prime \prime}(t_k)= - \big( \sin(t_k) + q_k + 2(1 - \cos(t_k))/t_k \big)$ and thus
$h_k^{\prime \prime}(t_k) < 0$ for $t_k < \pi$.

In order to see that this maximum is not only local,
note that $h_k(t) \to - \infty$ as $t \downarrow 0$
while $h_k(2\pi) = \sin(t_k) + q_k \cos(t_k) = h_k(t_k)$.
Without going into details, taking into account properties of
the functions defining $h_k$, we deduce that there is only one more
stationary point of $h_k$ in $(0,2\pi]$, corresponding to a local minimum.
Thus, the maximum at $t_k$ is global.

Finally, since $-t_k^2(1 - \cos(t_k)) < 0$ for $t_k \in (0,2\pi)$,
a point $(y_k,x)$ belongs to $\mathcal{H}_k$ if and only if
$x \geq 1/t_k$. Since $\gph S$ is the intersection of all
the closed halfspaces containing $G$ by \cite[Corollary 11.5.1]{Ro70},
we conclude that there is no $x_k \in S(y_k)$ with $x_k < 1/t_k$.
Hence, $S$ is not inner semicompact at $\yb$ wrt to $\dom S$.
\end{example}

We forgo the details of applying the above lemma
to a map whose graph is a union
of general convex sets, since the result inevitably suffers
from the limitations to the interiors of the domains of its components.
Instead, we look into polyhedral mappings.
Recall that a set $D\subset \R^s$ is said to be {\em convex polyhedral}
if it can be represented as the intersection of finitely many halfspaces.
We say that a set $E\subset \R^s$ is {\em polyhedral}
if it is a union of finitely many convex polyhedral sets.
A set-valued map is called (convex) polyhedral, if
its graph is a (convex) polyhedral set.

In the polyhedral setting, there is no problem with
the points on the boundary of the domain.
Indeed, the prominent result of Walkup and Wets \cite{WaWe69},
see also \cite[Example 9.35]{RoWe98}, says that
a convex polyhedral map $S$ is Lipschitz continuous on its domain,
i.e., there exist $\kappa > 0$ such that
\[S(y') \subset S(y) + \kappa \norm{y' - y}\B \quad \forall \, y,y' \in \dom S.\]
Since this property is obviously stronger even than inner calmness at any point,
we obtain the following result.
For the sake of completeness, we include also the well-known
result regarding (outer) calmness due to Robinson \cite{Rob81},
who used the name {\em upper Lipschitzness}.
We point out that the calmness below
is not localized to a point $\xb$,
since there is no neighbourhood $U$ as in \eqref{eq : CalmnessDef}.
\begin{theorem}[Two-sided calmness of polyhedral maps]
\label{the:2-sidedCalmnessPolyhedralMaps}
 Let $S:\R^m \rightrightarrows \R^n$ be a polyhedral set-valued map.
 Then there exists a number $\kappa > 0$
 such that for every $\yb \in \dom S$,
 $S$ is inner calm* with constant $\kappa$ wrt $\dom S$ at $\yb$
 as well as calm with constant $\kappa$ at $\yb$, i.e.,
 \[S(y) \subset S(\yb) + \kappa \norm{y - \yb}\B\]
 holds for all $y$ near $\yb$.
\end{theorem}

Note that, in inner calmness* we have found a suitable inner Lipschitzian property
which, from the Lipschitzness of convex polyhedral maps, extends to polyhedral maps.
The key reason is that inner calmness* is based on
inner semicompactness, which is preserved under finite unions.

\subsection{Multiplier mappings}

In \cite[Proposition 4.1]{GfrOut16}, Gfrerer and Outrata showed the following
interesting result.
\begin{proposition} \label{Pro:MSRMain}
 Let $M:\R^n\tto\R^m$ be a set-valued mapping having locally closed graph around $(x,\yb)\in\gph M$ and assume
 that $M$ is metrically subregular at $(x,\yb)$ with
 modulus $\kappa$. Then
 \begin{equation}\label{eq:NCtoM}
  N_{M^{-1}(\yb)}(x) \subset
  \big\{x^* \mv \exists y^* \in \kappa \norm{x^*} \B:
  (x^*,y^*) \in N_{\gph M}(x,\yb) \big\}.
 \end{equation}
\end{proposition}
\noindent
In fact, \cite[Proposition 4.1]{GfrOut16} contains,
apart from the additional result for the tangent cone,
the stronger, directional, estimate.
For our purposes, however, the important
part is the {\em bound} $\norm{y^*} \leq \kappa \norm{x^*}$.
Let us mention that the idea that metric subregularity
yields this bound appeared already in the proof of \cite[Theorem 4.1]{HenJouOut02}.
Moreover, similar arguments were also used in \cite[Lemma 3.2]{ChHi17}
in the setting of nonlinear programs.

\begin{corollary}\label{cor:inner_semicompactness}
 In the setting of Proposition \ref{Pro:MSRMain},
 metric subregularity of $M$ at $(x,\yb)$ implies
 inner semicompactness of the mapping
 \[\Lambda(x,x^*) := D^* M^{-1}(\yb,x)(-x^*)=
 \{y^* \mv (x^*,y^*) \in N_{\gph M}(x,\yb)\}
 \]
 at $(x,x^*)$ wrt $\gph N_{M^{-1}(\yb)}$ for every $x^* \in N_{M^{-1}(\yb)}(x)$.
\end{corollary}
\begin{proof}
 Suppose $(x_k,x_k^*) \to (x,x^*)$
 with $x_k^* \in N_{M^{-1}(\yb)}(x_k)$.
 Metric subregularity of $M$ at $(x,\yb)$ implies
 metric subregularity of $M$ at $(x_k,\yb)$ for
 all sufficiently large $k$ with the modulus $\kappa^{\prime}$
 independent of $k$ (see, e.g., \cite[Lemma 2]{Gfr14a} with $u=0$ and $\gamma=1$).
 Hence, Proposition \ref{Pro:MSRMain} yields
 the existence of $y_k^* \in \Lambda(x_k,x_k^*)$
 with $\norm{y_k^*} \leq \kappa^{\prime} \norm{x_k^*}$.
 In particular, by passing to a subsequence we may
 assume that $y_k^* \to y^*$ for some $y^*$,
 showing the inner semicompactness of $\Lambda$.
\end{proof}

For the constraint mapping $M(x)=\varphi(x) - Q$,
 where $Q \subset \R^m$ is a closed set and $\varphi: \R^n \to \R^m$
 is continuously differentiable, 
 \eqref{eq:NCtoM}
 gives the standard pre-image calculus rule
 \begin{equation}\label{eq:pre-image}
 N_C(x) \subset \nabla \varphi(x)^T N_Q(\varphi(x))
 \end{equation}
  where $C:=M^{-1}(0)=\varphi^{-1}(Q)$, see, e.g., \cite[Lemma 6.1]{BeGfrOut18a}.
 Moreover, denoting $\beta(x) := \nabla \varphi(x)$,
  the (multiplier) mapping $\Lambda$ attains the standard form
\[\Lambda(x,x^*) := \{\lambda \in N_{Q}(\varphi(x)) \mv \beta(x)^T \lambda = x^*\},\]
where we use $\lambda$ instead of $y^*$ to denote the multipliers.

\begin{remark}[On inner semicompactness of the multiplier mapping]
In the area of second-order analysis,
one often deals with a sequence $x_k^* \in N_{\varphi^{-1}(Q)}(x_k)$.
Under metric subregularity of $\varphi(x) - Q$, one gets the existence of suitable
multipliers $\lambda_k$, but needs to find also a limit multiplier
$\lambda \leftarrow \lambda_k$.
Hence, many authors used to assume boundedness of the multipliers,
for which they had to impose some stronger conditions,
such as the generalized Mangasarian-Fromovitz constraint qualification (GMFCQ),
see also \cite{gm15} for more subtle approach.
Thanks to the inner semicompactness of $\Lambda$
from Corollary \ref{cor:inner_semicompactness}, however,
we know that metric subregularity alone is sufficient for this task.
This fact has already been utilized in several recent works, see, e.g.,
\cite{BeGfrMor18,BeGfrOut18,GfrMo17a,GfrMo17}.
\end{remark}

We conclude this section by showing that $\Lambda$
is even inner calm* (in the fuzzy sense)
provided $Q$ is polyhedral.
The following lemma will be essential for our proof.

\begin{lemma}\label{lemma:Red_lemma_ext}
 Let $D$ be a convex polyhedral set and let $\bar z^* \in N_D(\bar z)$.
 Then there exists a neighbourhood $\mathcal{O}$ of $0$ such that
 \begin{equation}
   \big(\gph N_D - (\bar z,\bar z^*)\big) \cap \mathcal{O} = \{(w,w^*) \mv w \in \K_D(\bar z,\bar z^*), w^* \in (\K_D(\bar z,\bar z^*))^{\circ}, \skalp{w,w^*}=0\} \cap \mathcal{O},
 \end{equation}
  where $\K_D(\bar z,\bar z^*) := T_D(\bar z) \cap [\bar z^*]^\perp$
 stands for the {\em critical cone} to $D$ at $(\bar z,\bar z^*)$.
 \end{lemma}
\begin{proof}
 The {\em reduction lemma} \cite[Lemma 2E.4]{DoRo14}
 yields the existence of a neighbourhood $\mathcal{O}$ of $0$ such that
 \begin{equation}\label{eq:Reduction_Lemma}
   \big(\gph N_D - (\bar z,\bar z^*)\big) \cap \mathcal{O} = \gph N_{\K_D(\bar z,\bar z^*)} \cap \mathcal{O},
 \end{equation}
 and \cite[Proposition 2A.3]{DoRo14} gives
 the description of $\gph N_{\K_D(\bar z,\bar z^*)}$.
\end{proof}

\begin{theorem}\label{the:Main_Inner_calmness}
 Let $(x,x^*) \in \gph N_{C}$ for $C=\varphi^{-1}(Q)$
 with twice continuously differentiable $\varphi$
 and assume that the constraint mapping
 $M(x) = \varphi(x) - Q$ is metrically subregular at $(x,0)$.
 \begin{itemize}
  \item[(i)] If $Q$ is Clarke regular near $\varphi(x)$
  (which is the case if $Q$ is convex), then $\Lambda$ is inner semicompact at $(x,x^*)$
  wrt its domain.
  \item[(ii)] If $Q$ is convex polyhedral, then $\Lambda$ is inner calm* at $(x,x^*)$
  wrt its domain in the fuzzy sense.
  \item[(iii)] Given $(u,u^*) \in \Sp$, if $Q$ is convex polyhedral
  and the system $\varphi(\cdot) \in Q$ is {\em non-degenerate at} $x$ {\em in direction} $u$, i.e.,
  \begin{equation}\label{eq:DirNondeg}
    \beta(x)^T \lambda =0, \
    \lambda \in \Span N_{T_{Q}(\varphi(x))}(\nabla \varphi(x)u) \ \Longrightarrow \
    \lambda = 0,
  \end{equation}
  then $\Lambda$ is inner calm* at $(x,x^*)$ wrt its domain
  in direction $(u,u^*)$.
 \end{itemize}
\end{theorem}
\begin{proof}
 The first claim follows from Corollary
\ref{cor:inner_semicompactness} once we justify
 the use of $\dom \Lambda$ instead of $\gph N_{C}$.
 Due to the Clarke regularity, however, we
 obtain $N_{Q}(\varphi(x^{\prime}))=\widehat{N}_{Q}(\varphi(x^{\prime}))$ for $x^{\prime}$ near $x$
 and \cite[Theorem 6.14]{RoWe98} then yields that
 \[\beta(x^{\prime})^T N_Q(\varphi(x^{\prime})) = \beta(x^{\prime})^T \widehat{N}_Q(\varphi(x^{\prime})) \subset N_C(x^{\prime}).\]
 Moreover, the opposite inclusion is also true, taking into account
 \eqref{eq:pre-image} and the fact that $M$ is subregular also at $(x^{\prime},0)$
 as argued in the proof of Corollary \ref{cor:inner_semicompactness}.
 Hence, locally around $(x,x^*)$, one has $\dom \Lambda = \gph N_{C}$.
 
 To show the second one, assume that $Q$ is convex polyhedral
 and consider a sequence $(x_k,x_k^*) \in \dom \Lambda$
 converging to $(x,x^*)$ from a direction $(u,u^*) \in \Sp$,
 i.e.,
 \[(x_k,x_k^*) = (x,x^*) + t_k (u_k,u_k^*)\]
 for some $t_k \downarrow 0$ and $(u_k,u_k^*) \to (u,u^*)$.
 The inner semicompactness of $\Lambda$ wrt $\dom \Lambda$
 from (i) yields the existence of
 $\tilde \lambda_k \in N_Q(\varphi(x + t_k u_k))$ with
 \begin{equation*}\label{eq:seq}
  x^* + t_k u_k^* = \beta(x + t_k u_k)^T \tilde \lambda_k
 \end{equation*}
 as well as $\lambda \in N_Q(\varphi(x))$ with
 $x^*= \beta(x)^T \lambda$ such that $\tilde \lambda_k \to \lambda$.
 Moreover, the Taylor expansion gives
 \begin{equation}\label{eq : TayExp}
  \beta(x)^T \frac{\tilde \lambda_k - \lambda}{t_k} = u_k^* - \nabla \skalp{\lambda,\beta}(x)u + o(1).
 \end{equation}
 
 Employing Lemma \ref{lemma:Red_lemma_ext}
 and denoting $\K := \K_Q(\varphi(x),\lambda)$
 we infer
 \begin{equation}\label{eq:red_lemma_out}
   \varphi(x + t_k u_k) - \varphi(x) \in \K, \ \
 \tilde \lambda_k - \lambda \in \K^{\circ}, \ \
 \skalp{\varphi(x + t_k u_k) - \varphi(x), \tilde \lambda_k - \lambda} = 0
 \end{equation}
 for sufficiently large $k$.
 Thus, $\F_k := \K \cap [\tilde \lambda_k - \lambda]^{\perp}$
 is a {\em face} of the critical cone $\K$, see \cite[p. 258]{DoRo14}, with
 $\varphi(x + t_k u_k) - \varphi(x) \in \F_k$.
 Since there are only finitely many faces of a polyhedral cone, we may assume that $\F_k \equiv \F$ and, in particular,
 we have
 \[(\tilde \lambda_k - \lambda) / t_k \in \K^{\circ} \cap (\Span \F)^{\perp}.\]
 \if{
 The polyhedrality of $Q$ yields the polyhedrality
 of $\gph N_Q$ and hence
 \[\left(\frac{\varphi(x + t_k u_k) - \varphi(x)}{t_k}, \frac{\tilde \lambda_k - \lambda}{t_k} \right) \in T_{\gph N_Q}(\varphi(x),\lambda) = \gph N_{\K_Q(\varphi(x),\lambda)},\]
 by the reduction lemma formula \eqref{EqTangConeGraphNormalCone}.
 This means that
 \begin{eqnarray*}
 &&(\varphi(x + t_k u_k) - \varphi(x))/t_k \in \K_Q(\varphi(x),\lambda), \
 (\tilde \lambda_k - \lambda) / t_k \in \big(\K_Q(\varphi(x),\lambda)\big)^{\circ}, \\
 && \skalp{\varphi(x + t_k u_k) - \varphi(x), (\tilde \lambda_k - \lambda) / t_k} = 0.
 \end{eqnarray*}
 Hence, for every $k$ there exists a face $\F_k := \K_Q(\varphi(x),\lambda) \cap
 [(\tilde \lambda_k - \lambda) / t_k]^{\perp}$ of the critical cone with
 $\varphi(x + t_k u_k) - \varphi(x) \in \F_k$ and, due to the finiteness
 of faces of a polyhedral cone, we may assume that $\F_k \equiv \F$
 with
 \[(\tilde \lambda_k - \lambda) / t_k \in (\K_Q(\varphi(x),\lambda))^{\circ} \cap (\Span \F)^{\perp}.\]
 }\fi
 
 Hence, we can now invoke Hoffman's lemma \cite[Theorem 2.200]{BonSh00} to find for every $k$ some $\eta_k\in \K^{\circ} \cap (\Span \F)^{\perp}$ satisfying
\[\beta(x)^T\eta_k=\beta(x)^T\frac{\tilde\lambda_k-\lambda}{t_k}\]
 and $\norm{\eta_k} \leq \alpha\norm{\beta(x)^T(\tilde\lambda_k-\lambda)/{t_k}}$
for some constant $\alpha > 0$ not depending on $k$. Since the right-hand side of \eqref{eq : TayExp} is bounded, so is $\eta_k$ and by possibly passing to a subsequence we can assume that $\eta_k$ converges to some $\eta \in \K^{\circ} \cap (\Span \F)^{\perp}$ satisfying
\[\beta(x)^T\eta= u^* - \nabla \skalp{\lambda,\beta}(x)u\]
by \eqref{eq : TayExp}.
Since $\eta_k \in \K^{\circ} \cap (\Span \F)^{\perp}$
we see that \eqref{eq:red_lemma_out} remains true
with $\tilde \lambda_k - \lambda$ replaced by $t_k \eta_k$.
Consequently, Lemma \ref{lemma:Red_lemma_ext} yields that
for sufficiently large $k$ we have
\[(\varphi(x + t_k u_k),\lambda + t_k \eta_k) = (\varphi(x),\lambda)
+ (\varphi(x + t_k u_k) - \varphi(x), t_k \eta_k) \in \gph N_{Q}.\]
Hence, setting $\lambda_k := \lambda + t_k \eta_k$,
for any $\kappa_{(u,u^*)} > \norm{\eta}$ we get
\begin{equation}\label{eq:icEstProof}
\norm{\lambda_k - \lambda} = \norm{\eta_k} t_k =
\frac{\norm{\eta_k}}{\norm{(u_k,u_k^*)}} \norm{(x_k,x_k^*) - (x,x^*)}
\leq \kappa_{(u,u^*)} \norm{(x_k,x_k^*) - (x,x^*)}
\end{equation}
for $k$ large enough, since $\norm{(u_k,u_k^*)} \to 1$ and $\norm{\eta_k} \to \norm{\eta}$.

Finally, denoting
\[\tilde x_k^* := \beta(x + t_k u_k)^T \lambda_k =
\beta(x)^T \lambda + t_k(\beta(x)^T \eta_k + \nabla \skalp{\lambda,\beta}(x)u) + o(t_k)
= x^* + t_k u_k^* + o(t_k) = x_k^* + o(t_k)\]
yields, in particular, that $(x_k,\tilde x_k^*)$ converges to $(x,x^*)$ from direction $(u,u^*)$.
Hence, the inner calmness* of $\Lambda$ in direction $(u,u^*)$ in the fuzzy sense
with the modulus not greater than $\norm{\eta}$ follows.
Since direction $(u,u^*)$ was arbitrary, Lemma \ref{Lem:DirVsStand}
yields the inner calmness* in the fuzzy sense of $\Lambda$.

The third statement can now be shown easily.
If $x_k=x$ for infinitely many $k$,
in the previous argument we actually get
\[\beta(x)^T\eta_k=\beta(x)^T\frac{\tilde\lambda_k-\lambda}{t_k} = u_k^*.\]
Then, however, we obtain
\[\beta(x)^T\lambda_k = x^* + t_k u_k^* = x_k^*,\]
which means that the new multipliers $\lambda_k$ correspond to the same $x_k^*$
and this implies the inner calmness* in direction $(u,u^*)$.

On the other hand, if $x_k \neq x$ for infinitely many $k$,
we conduct the same arguments as in case (ii)
and then show that we can in fact use the original multipliers $\tilde \lambda_k$.
To this end, note that
$\tilde \lambda_k - \lambda_k \in \Span N_Q(\varphi(x + t_k u_k))$
and thus \cite[Proposition 2.14]{BeGfrOut18} says that the assumed directional non-degeneracy yields the existence of $\alpha > 0$ such that
\begin{equation*}
    \norm{\tilde \lambda_k - \lambda_k} \leq
    \alpha \norm{\beta(x + t_k u_k)^T(\tilde \lambda_k - \lambda_k)} =
    \alpha \norm{x_k^* - \tilde x_k^*} = o(t_k).
\end{equation*}
Consequently, we obtain
\begin{eqnarray*}
    \norm{\tilde \lambda_k - \lambda} \leq
    \norm{\lambda_k - \lambda} + \norm{\tilde \lambda_k - \lambda_k} \leq
    \norm{\eta_k} t_k  + o(t_k) = (\norm{\eta_k} + o(1)) \, t_k
\end{eqnarray*}
and we infer that $\tilde \lambda_k$ also satisfies the inner calmness* estimate
\eqref{eq:icEstProof} for any $\kappa_{(u,u^*)} > \norm{\eta}$ due to $\norm{\eta_k} + o(1) \to \norm{\eta}$.
\end{proof}
\begin{remark}\label{Rem:Main}
 Let us briefly comment on the above results.
 \begin{enumerate}
  \item We only used the fact that $\beta(x) := \nabla \varphi(x)$
 in the proof of (iii).
 More precisely, we only used the feature of non-degeneracy
 from \cite[Proposition 2.14]{BeGfrOut18}.
 Looking into the proof of this proposition, however, it is clear
 that the result remains valid under assumption \eqref{eq:DirNondeg}
 even if $\beta(x) \neq \nabla \varphi(x)$.
 In such case, it is not appropriate to call this assumption
 non-degeneracy of $\varphi(\cdot) \in Q$, naturally.
 
 This means that for arbitrary sufficiently smooth
 function $\beta(x)$ it holds that $\Lambda(x,x^*)$
 is inner calm* at $(x,x^*)$ wrt its domain in the fuzzy sense, provided
 $\Lambda$ is inner semicompact at $(x,x^*)$ wrt $\gph N_{C}$
 and $Q$ is convex polyhedral and it is even inner calm* in direction
 $(u,u^*)$ if \eqref{eq:DirNondeg} holds as well.
 This enables us to handle the parametrized setting
 in Section 5 with ease.
 \item As already mentioned, in order to replace the inner semicompactness in (i)
 by local boundedness, one has
 to impose a stronger assumption, such as GMFCQ.
 Moreover, in order to obtain analogous results in terms of
 inner semicontinuity and inner calmness,
 one can, e.g., assume uniqueness of the multiplier.
 In turn, one needs to strengthen the assumption even more
 and require, say, the standard non-degeneracy from \cite[Formula 4.17]{BonSh00},
 which corresponds to \eqref{eq:DirNondeg} with $u = 0$.
 Directional non-degeneracy was introduced in \cite[Definition 2.13]{BeGfrOut18}
 and it is strictly milder than the standard one, see the example below.
 \item If $(u,u^*)=(0,u^*) \in \Sp$ in (iii), \eqref{eq:DirNondeg}
 becomes the standard non-degeneracy assumption.
 Since $u^* \in DN_{C}(x,x^*)(u)$, such directions exist if and only if
 the normal cone mapping $N_C$ is not isolatedly calm at $(x,x^*)$.
 \item Finally, note that our proof follows the arguments
 from the proof of \cite[Theorem 5.3]{BeGfrOut18}.
 \end{enumerate}
\end{remark}

The following example comes from \cite[Example 2.15]{BeGfrOut18}.
\begin{example}\label{Ex}
Consider the constraint system $\varphi(x) \leq 0$ given by
\[x_1 - x_4 \leq 0,\ -x_1 - x_4 \leq 0,\ x_2 - x_4 \leq 0,\ -x_2 - x_4 \leq 0,\
x_3 + x_1^2 - x_4 \leq 0,\ -x_3 - x_4 \leq 0\]
and the point $\bar x=0$.
We have $\ker \nabla\varphi(\xb)^T = \R(1,1,0,0,-1,-1)^T + \R(0,0,1,1,-1,-1)^T$.
Denoting $Q:=\R_{-}^6$ yields $N_{Q}(\varphi(\xb))=\R_{+}^6$
and so $\Span N_{Q}(\varphi(\xb))=\R^6$,
implying that the system is degenerate.
Moreover, the so-called {\em constant rank constraint qualification}
is clearly also violated at $\xb$.

On the other hand, MFCQ holds $\xb$, as can be seen from
$\ker \nabla\varphi(\xb)^T \cap N_{Q}(\varphi(\xb)) = \{0\}$.
More interestingly, however, let us explain that the system is actually non-degenerate
in every non-zero direction.
Given a direction $u$ with $\nabla \varphi(\xb) u \leq 0$
and denoting $\mathcal{I}(u) := \{i \mv \nabla \varphi_i(\xb)u = 0\}$, we have
$\Span N_{T_{Q}(\varphi(x))}(\nabla \varphi(x)u) =
\{\lambda \in \R^6 \mv \lambda_i = 0, i \notin \mathcal{I}(u)\}$.
Thus, non-degeneracy in direction $u$ is equivalent to the linear independence of the gradients $\nabla \varphi_i(\xb),\, i \in \mathcal{I}(u)$.
For $u \neq 0$, $\nabla \varphi(\xb) u \leq 0$ readily implies that $u_4 > 0$,
since $u_4 < 0$ is not possible and $u_4 = 0$ enforces $u = 0$.
Hence, $\mathcal{I}(u)$ can never contain $1$ and $2$ simultaneously,
$3$ and $4$ simultaneously or $5$ and $6$ simultaneously,
which ensures the linear independence of $\nabla \varphi_i(\xb),\, i \in \mathcal{I}(u)$.

Consequently, the above theorem yields that
$\Lambda$ is inner calm* at $(\xb,\xb^*)$ wrt its domain
for any $\xb^* \in N_{\varphi^{-1}(Q)}(\xb)$.
\end{example}

\section{Selected calculus rules}

In this section, we show some new calculus
rules based on (fuzzy) inner calmness*
with focus on the primal objects
(tangents and graphical derivatives).

\subsection{Tangents and directional normals to image sets}

In this section, we deal with the rules for image sets,
which provide the base for all of the remaining calculus.
Given a closed set $C \subset \R^n$ and a continuous mapping $\varphi: \R^n \to \R^l$,
set $Q := \varphi(C)$ and consider $\yb \in Q$.
Moreover, let $\Psi : \R^l \rightrightarrows \R^n$ be given by
$\Psi(y) := \varphi^{-1}(y) \cap C$ and note that
$\dom \Psi = Q$ and $\gph \Psi = \gph \varphi^{-1} \cap (\R^l \times C)$.
Hence $\gph \Psi$
is closed by the properties of $\varphi$ and $C$
and, recalling Lemma \ref{Lem : iscomp_dom_closed},
we suppose that $\Psi$ is also inner semicompact at $\yb$
wrt $\dom \Psi$.

\begin{theorem}[Tangents to image sets]\label{The : ImageSet_Tangents}
  If $\varphi: \R^n \to \R^l$ is calm at some $\xb \in \Psi(\yb)$, we have
  \begin{equation*}
   T_C(\xb) \subset \{ u \mv \exists\, v \in T_Q(\yb) \textrm{ with } v \in D\varphi(\xb)u \}.
  \end{equation*}
  On the other hand, if $\Psi$ is inner calm* at $\yb$
  wrt $\dom \Psi$ in the fuzzy sense, then
  \begin{equation*}
   T_Q(\yb) \subset \bigcup_{\bar x \in \Psi(\bar y)} \{v \mv \exists\, u \in T_C(\xb) \textrm{ with } v \in D\varphi(\xb)u \}
  \end{equation*}
  and, moreover, if there exists $\xb \in \Psi(\yb)$ such that $\Psi$ is inner calm at $(\yb,\xb)$ wrt $\dom \Psi$,
  then the estimate holds with this $\xb$, i.e., the union over $\Psi(\yb)$ is superfluous.
\end{theorem}
\begin{proof}
 Pick $u \in T_C(\xb)$. We will show a slightly stronger statement,
 namely that $D\varphi(\xb)u \cap T_Q(\yb) \neq \emptyset$,
 provided $\varphi$ is calm at $\xb$ {\em in direction} $u$.
 Indeed, consider $t_k \downarrow 0$ and $u_k \to u$ with
 $\xb + t_k u_k \in C$ and observe that
 \begin{equation} \label{eq_Q}
 Q \ni \varphi(\xb + t_k u_k) = \bar y + t_k v_k,
 \end{equation}
 where $v_k := (\varphi(\xb + t_k u_k) - \bar y)/t_k$
 is bounded by the assumed calmness in direction $u$
 and we may assume that $v_k \to v$ for some $v \in T_Q(\yb)$.
 Moreover, \eqref{eq_Q} can be written as
 $(\xb + t_k u_k,\bar y + t_k v_k) \in \gph \varphi$,
 showing also $v \in D\varphi(\xb)u$.
 
 Consider now $v \in T_Q(\yb)$.
 The inner calmness* of $\Psi$ at $\bar y$ wrt
 $\dom \Psi=Q$ in the fuzzy sense
 {\em in direction} $v$ yields the existence of
 $t_k \downarrow 0$, $v_k \to v$ with $\yb + t_k v_k \in Q$
 as well as $x_k \in C$ and $\xb$ with
 \[\bar y + t_k v_k = \varphi(x_k) \ \textrm{ and } \
 \norm{x_k - \xb} \leq \kappa_v t_k \norm{v_k}\]
 for some $\kappa_v > 0$.
 This means, however, that
 $(\xb + t_k u_k, \yb + t_k v_k) \in \gph \varphi$
 for the bounded sequence $u_k := (x_k - \xb)/t_k$
 and the existence of $u \in T_C(\xb)$ with
 $v \in D\varphi(\xb)u$ follows.
 Finally, the closedness of $C$ and the
 continuity of $\varphi$ imply $\bar x \in \Psi(\bar y)$.
 
 The last statement now follows easily
 from the definition of inner calmness.
\end{proof}

If $\varphi$ is differentiable, $D\varphi(\xb)u = \{\nabla\varphi(\xb) u\}$
is a singleton and we obtain the following simpler estimates.

\begin{corollary}[Tangents to image sets - differentiable case]\label{Cor:T_I_S_dif}
 If $\varphi$ is continuously differentiable, then
 \[T_Q(\bar y) \supset \bigcup_{\bar x \in \Psi(\bar y)}
 \nabla \varphi(\bar x) T_C(\bar x).\]
 The above inclusion becomes equality if $\Psi$
 is inner calm* at $\yb$ wrt $\dom \Psi$ in the fuzzy sense.
\end{corollary}

In \cite[Theorem 6.43]{RoWe98}, the upper estimate
for $T_Q(\yb)$ is obtained only in quite a special setting,
namely for $C$ convex and $\varphi$ linear.
As the previous section shows, our approach is
applicable in a much broader context.
On the other hand, our result also does not cover
the estimate from \cite[Theorem 6.43]{RoWe98},
which reads
\begin{equation}\label{eq:TC_forward_Convex}
T_Q(\yb) = \cl \big( \varphi(T_C(\xb)) \big)
\textrm{ for every } \xb \in \varphi^{-1}(\yb) \cap C.
\end{equation}
To see this, consider again the mapping $S$
from Example \ref{Ex_isc_violated} and set $C=\gph S$, $Q = \dom S$,
i.e., $Q=\varphi(C)$, where $\varphi$ is just the projection:
$\varphi(x)=y$ for $x=(y,z) \in \R^2 \times \R$.
While this setting fits into \cite[Theorem 6.43]{RoWe98},
Corollary \ref{Cor:T_I_S_dif} cannot be applied.
Indeed, among the standing assumptions,
we ask $\Psi$ to be inner semicompact at $\yb$,
but it is not since $\Psi(y) = \{y\} \times S(y)$ and $S$
is not inner semicompact at $\yb$, as clarified in Example \ref{Ex_isc_violated}.

Let us look more closely into the two approaches.
Since it is not essential for our argument, however, we omit the details.
For $\bar y = (1,0)$, we clearly have
$T_Q(\yb) = \R_- \times \R$, since $Q$ is the closed unit sphere.
On the other hand, one can show that $S(\yb)= \R_+$, and thus
$\Psi(\yb) = \varphi^{-1}(\yb) \cap C = \{\yb\} \times \R_+$,
and
\[
\varphi(T_C(\yb,\zb)) =
\{v \in \R^2 \mv \exists w \in \R: (v,w) \in T_C(\yb,\zb)\}
= \{(0,0)\} \cup \{v \in \R^2 \mv v_1 < 0\}
\]
for any $\zb \in \R_+$.
This shows the importance of the closure in \eqref{eq:TC_forward_Convex}.
Moreover, it also shows that the estimate from Corollary \ref{Cor:T_I_S_dif}
is actually false, since vectors $(0,\pm 1) \in T_Q(\yb)$
do not belong to $\nabla \varphi(\bar x) T_C(\bar x) = \varphi(T_C(\xb))$ for any $\xb=(\yb,\zb) \in \Psi(\yb)$.

Naturally, the estimates for directional normals to image sets
from \cite[Theorem 3.2]{BeGfrOut18a} can also be enriched
by the inner calmness* assumption.
The next theorem shows how these estimates differ
based on the assumption used.
Given a direction $v \in \R^l$, let us denote
\begin{equation*}
 \Sigma(v):= \{y^* \mv D^*\varphi(\xb;(u,v))(y^*) \cap N_C(\xb;u) \neq \emptyset\}.
\end{equation*}
Trivial modifications of the proof of \cite[Theorem 3.2]{BeGfrOut18a}
yield the following result.
\begin{theorem}[Directional normals to image sets] \label{The : ConstSetForw}
 Consider a direction $v \in \R^l$ and
 assume that $\varphi$ is Lipschitz continuous near every $\xb \in \Psi(\yb)$.
 Then
  \begin{itemize}
  \item[(i)] if $\Psi$ is inner semicompact at $\yb$ wrt $\dom \Psi$
  in direction $v$, one has
  \begin{equation*}
   N_Q(\yb;v) \subset \bigcup\limits_{\xb \in \Psi(\yb)} \Big(
   \bigcup\limits_{u \in D\Psi(\yb,\xb)(v)}
   \Sigma(v) \ \ \cup
   \bigcup\limits_{u \in D\Psi(\yb,\xb)(0) \cap \Sp}
   \Sigma(0) \Big)
  \end{equation*}
  and the union over $\Psi(\yb)$ is superfluous
  if $\Psi$ is inner semicontinuous at $(\yb,\xb)$
  wrt $\dom \Psi$ in $v$;
  \item[(ii)] if $\Psi$ is inner calm* at $\yb$ wrt $\dom \Psi$
  in $v$, one has
  \begin{equation*}
   N_Q(\yb;v) \subset \bigcup\limits_{\xb \in \Psi(\yb)} \
   \bigcup\limits_{u \in D\Psi(\yb,\xb)(v)}
   \{y^* \mv D^*\varphi(\xb;(u,v))(y^*) \cap N_C(\xb;u) \neq \emptyset\}
  \end{equation*}    
  and the union over $\Psi(\yb)$ is superfluous
  if $\Psi$ is inner calm at $(\yb,\xb)$
  wrt $\dom \Psi$ in $v$.
  \end{itemize}
\end{theorem}
\noindent
In fact, $\varphi$ needs to be Lipschitz continuous only in
relevant directions, see \cite[Theorem 3.2]{BeGfrOut18a}.
The following example shows that inner calmness* in the fuzzy sense is not 
enough for (ii).

\begin{example}
Let $C \in \R^3$ be given by
\[C=\R_- \times \R \times \{0\} \ \cup \
\{x=(x_1,x_2,x_3) \mv x_1 > 0,\, x_2=\sqrt{x_1},\, x_3=\sqrt{x_2}\},\]
let $\varphi:\R^3 \to \R^2$ be given by
$\varphi(x)=(x_1,x_2)$ and set
\[Q:=\varphi(C) = \R_- \times \R \ \cup \
\{(x_1,x_2) \mv x_1 > 0,\, x_2=\sqrt{x_1}\}.\]
In order to be consistent with our notation, we use $y$ for $(x_1,x_2)$.
First, we claim that
\[\Psi(y)=\varphi^{-1}(y) \cap C=
\{x=(y,x_3) \mv x \in C\}\]
is not inner calm* at $\yb=(0,0)$ in direction $v=(0,1)$ wrt $Q$,
but in the fuzzy sense, it is.

Indeed, the sequence $Q \ni y^k = (1/k^2,1/k)=(0,0) + 1/k(1/k,1)$
converges to $\yb$ from $v$, but
$\Psi(y^k)=(1/k^2,1/k,1/\sqrt{k}) \to (0,0,0)=\Psi(\yb)=\xb$ and thus
\[\norm{\Psi(y^k) - \Psi(\yb)}/\norm{y^k - \yb} \geq \sqrt{k} \to \infty.\]
On the other hand, the sequence
$Q \ni \tilde y^k = (0,1/k)=(0,0) + 1/k(0,1)$
also converges to $\yb$ from $v$ with $\Psi(\tilde y^k)=\Psi(\yb)$.

We have $y^* = (-1,0) \in N_Q(\yb,v)$. The only direction
$u \in T_C(\xb) = \R_- \times \R \times \{0\} \ \cup \
\{0\} \times \{0\} \times \R_+$ with $\nabla \varphi(\xb) u = v$,
however, is $u=(0,1,0)$ and $\nabla \varphi(\xb)^T y^*
= (-1,0,0) \notin N_C(\xb;u) = \R_+ \times \{0\} \times \R$.
This shows that the estimate from (ii) does not hold.
\end{example}

In the remaining part of this section we will use the above
results to derive other calculus rules for
graphical derivatives of set-valued mappings.
Let us just mention that one can also obtain
the corresponding estimates, based on inner calmness*,
for the dual constructions defined via directional normals, 
such as the directional subdifferentials of
{\em value} (or {\em marginal}) functions, see \cite[Theorem 4.2]{BeGfrOut18a},
or the directional coderivatives of compositions
or sums of set-valued maps, see \cite[Theorems 5.1 and 5.2]{BeGfrOut18a}.

\subsection{Chain rule and sum rule for graphical derivatives of set-valued maps}

Let us begin with the chain rule. To this end,
consider the mappings $S_{1}: \mathbb{R}^{n}\rightrightarrows \mathbb{R}^{m},
S_{2}:\mathbb{R}^{m}\rightrightarrows \mathbb{R}^{s}$
with closed graphs, i.e., $S_{1}$ and $S_{2}$ are osc,
and set $S=S_{2}\circ S_{1}$.
Let $(\xb,\bar z) \in \gph S$ and $(u,w)\in \mathbb{R}^{n} \times \mathbb{R}^{s}$ be a pair of directions.
Finally, consider
the ``intermediate'' map $\Xi: \mathbb{R}^{n} \times \mathbb{R}^{s} \rightrightarrows \mathbb{R}^{m}$ defined by
\[
\Xi(x,z):= S_{1}(x) \cap S_{2}^{-1}(z) =\{y\in S_{1}(x) \mv z \in S_{2}(y)\}.
\]

Following the approach from \cite[Theorem 10.37]{RoWe98} one has
\begin{equation}\label{eq:BasicRel}
\dom \Xi = \gph S = \varphi(\gph \Xi), \quad \gph \Xi = \phi^{-1}(\gph S_{1} \times \gph S_{2}),
\end{equation}
where $\varphi:(x,z,y)\mapsto (x,z)$ and $\phi:(x,z,y)\mapsto (x,y,y,z)$.
This suggests applying first the image set rule from Corollary
\ref{Cor:T_I_S_dif} and then the pre-image set rule.

Note also that
\[
 \Psi(x,z) := \varphi^{-1}(x,z) \cap \gph \Xi
= \{(x,z,y) \mv y \in \Xi(x,z)\},
\]
so it is not surprising that the fuzzy inner calmness* assumption,
needed for the image set rule, can be imposed on $\Xi$, instead of $\Psi$.
As before, we note that $\gph \Xi$ is closed
by continuity of $\phi$ and osc of $S_1$ and $S_2$.
Taking into account Lemma \ref{Lem : iscomp_dom_closed},
let us assume that $\Xi$ is inner semicompact at
$(\xb,\bar z)$ wrt $\dom \Xi$.

On the other hand, for the pre-image set rule we will need
metric subregularity of the mapping
\begin{equation*}\label{eq-60}
 F(x,z,y):= \gph S_{1} \times \gph S_{2} - \phi(x,z,y) = \left [
 \begin{array}{l}
 \gph S_{1}-(x,y)\\
 \gph S_{2}-(y,z)
 \end{array}
 \right ].
\end{equation*}

\begin{theorem}[Chain rule for graphical derivatives] \label{The : ChainRuleGder}
If $\Xi$ is inner calm* at $(\xb,\zb)$ wrt $\dom \Xi$ in the fuzzy sense, then
\begin{equation}\label{eq:GDChRule}
 DS(\xb,\bar z) \subset \bigcup\limits_{ \yb \in \Xi(\bar{x},\bar{z}) } DS_2(\yb,\bar z) \circ DS_1(\xb,\yb).
\end{equation}
On the other hand, for any $\yb \in \Xi(\bar{x},\bar{z})$,
\begin{equation}\label{eq:GDChRuleRev}
 DS_2(\yb,\bar z) \circ DS_1(\xb,\yb) \subset DS(\xb,\bar z)
\end{equation}
holds provided $F$ is metrically subregular at $((\xb,\yb,\bar z),(0,0))$ and
\begin{equation}\label{eq: TangProdRel}
(u,v,v,w) \in T_{\gph S_{1} \times \gph S_{2}}(\xb,\yb,\yb,\bar z) \ \iff \ (u,v) \in T_{\gph S_{1}}(\xb,\yb), \, (v,w) \in T_{\gph S_{2}}(\yb,\bar z).
\end{equation}
\end{theorem}
\begin{proof}
Instead of proving the two estimates one by one,
we rather propose a proof that emphasizes the role of
the graphical derivative of $\Xi$.
To this end, recall \eqref{eq:BasicRel}.
First, we claim that, given $\yb \in \Xi(\bar{x},\bar{z})$ and $v \in D\Xi((\xb,\zb),\yb)(u,w)$, we have
 \[w \in DS(\xb,\bar z)(u) \ \textrm{ and } \
  w \in DS_2(\yb,\bar z) \circ DS_1(\xb,\yb)(u).\]

Indeed, since $(u,w) = \nabla \varphi(\xb,\yb,\zb)(u,v,w)$,
Corollary \ref{Cor:T_I_S_dif} yields $w \in DS(\xb,\bar z)(u)$.
On the other hand, $\nabla \phi(\xb,\yb,\zb)(u,v,w)=(u,v,v,w)$ and
\cite[Theorem 6.31]{RoWe98} implies
$(u,v,v,w) \in T_{\gph S_{1} \times \gph S_{2}}(\xb,\yb,\yb,\bar z)$.
Hence, $w \in DS_2(\yb,\bar z) \circ DS_1(\xb,\yb)(u)$
follows from the forward implication $\Rightarrow$ of \eqref{eq: TangProdRel},
which always holds by \cite[Proposition 6.41]{RoWe98}.

Note that we have used the image set and the pre-image set rule,
but we have not needed any of the assumptions.
The assumptions are in fact needed to get
$\yb \in \Xi(\bar{x},\bar{z})$ and $v \in D\Xi((\xb,\zb),\yb)(u,w)$.

Consider $w \in DS(\xb,\bar z)(u)$.
Thanks to the fuzzy inner calmness* of $\Xi$,
Corollary \ref{Cor:T_I_S_dif} precisely gives
$\yb \in \Xi(\bar{x},\bar{z})$ and $v \in D\Xi((\xb,\zb),\yb)(u,w)$
and \eqref{eq:GDChRule} follows.

Given now $w \in DS_2(\yb,\bar z) \circ DS_1(\xb,\yb)(u)$,
using \eqref{eq: TangProdRel} and then the subregularity of $F$
(which is equivalent to the calmness of $F^{-1}$),
\cite[Proposition 1]{HenOut05} again yields
$v \in D\Xi((\xb,\zb),\yb)(u,w)$.
\end{proof}
\noindent
We point out that the subregularity assumption and \eqref{eq: TangProdRel}
can be replaced by asking directly for the implication
\begin{equation}\label{eq: TangProdRel*}
 w \in DS_2(\yb,\bar z) \circ DS_1(\xb,\yb)(u) \ \Longrightarrow \
\exists \, v \in D\Xi((\xb,\zb),\yb)(u,w).
\end{equation}
Thus, if this is satisfied for every $\yb \in \Xi(\bar{x},\bar{z})$,
we get the exact chain rule
\begin{equation*}
 DS(\xb,\bar z) = \bigcup\limits_{ \yb \in \Xi(\bar{x},\bar{z}) } DS_2(\yb,\bar z) \circ DS_1(\xb,\yb).
\end{equation*}

Note that in \cite[p. 454]{RoWe98}, the authors argue that it is difficult
to obtain a reasonable chain rule for graphical derivatives,
since the image set and pre-image set rules for tangent cones
in general work in ``opposite direction'',
see \cite[Theorems 6.31 and 6.43]{RoWe98}.
More precisely, the upper estimate of the tangent cone
to image set in \cite[Theorem 6.43]{RoWe98}
was obtained only under quite restrictive assumptions,
see the discussion after Corollary \ref{Cor:T_I_S_dif}.
Here we see that it can be done using (fuzzy) inner calmness*,
which can be applied in relevant situations, as the previous
(and even more so the next) section shows.

Next we briefly discuss the situation when one of the mappings
is single-valued.

\begin{corollary}\label{Cor:S_1_single}
If $S_1$ is single-valued and calm at $\xb$, then
$DS(\xb,\bar z) \subset DS_2(S_1(\xb),\bar z) \circ DS_1(\xb)$.
If $S_1$ is even differentiable at $\xb$, the opposite inclusion
holds as well provided the map
 \begin{equation*}
     (x,z,y) \tto (S_1(x) - y) \times (\gph S_{2} - (y,z))
 \end{equation*}
is metrically subregular at $((\xb,\zb,S_1(\xb)),(0,0,0))$.
\end{corollary}
\begin{proof}
If $S_1$ is single-valued and calm at $\xb$, then $\Xi$,
given by $\Xi(x,z) = S_1(x)$ if $z \in S_2(S_1(x))$ and
$\Xi(x,z) = \emptyset$ otherwise,
is single- (or empty-) valued and inner calm at $(\xb,\zb) \in \gph S$.
This shows the first statement.

To prove the second, let us first argue that \eqref{eq: TangProdRel}
holds for differentiable $S_1$. Indeed, given $(u,v,w)$ with
$v=\nabla S_1(\xb)(u)$ and $(v,w) \in T_{\gph S_{2}}(S_1(\xb),\zb)$,
consider $t_k \downarrow 0$, $(v_k,w_k) \to (v,w)$
with $\zb + t_k w_k \in S_2(S_1(\xb) + t_k v_k)$.
We get $S_{1}(\xb + t_k u) = S_1(\xb) + t_k \tilde v_k$
 for $\tilde v_k = \nabla S_{1}(\xb)u + o(1) \to v$.
 Thus
 \[
 (\xb,S_1(\xb),S_1(\xb),\zb) + t_k (u,\tilde v_k,v_k,w_k) \in
 \gph S_1 \times \gph S_2
 \]
 and \eqref{eq: TangProdRel} follows.
Finally, it is known and easy to check that the
assumed subregularity is equivalent to
the subregularity required by Theorem \ref{The : ChainRuleGder}
(instead of $\gph S_1 - (x,y)$ we use $S_1(x) - y$).
\end{proof}

\begin{corollary}\label{Cor:S_2_single}
Assume that $S_2$ is single-valued. 
Then \eqref{eq:GDChRule} holds if $\Xi$ is inner calm* at $(\bar{x},\bar{z})$
wrt $\dom \Xi$ in the fuzzy sense.
On the other hand, given $\yb \in \Xi(\bar{x},\bar{z})$,
\eqref{eq:GDChRuleRev} holds if $S_2$ is differentiable at $\yb$.
\end{corollary}
\begin{proof}
We show that if $S_2$ is differentiable at $\yb$,
\eqref{eq: TangProdRel*} holds.
Consider $(u,v,w)$ with $(u,v) \in T_{\gph S_{1}}(\xb,\yb)$
and $w=\nabla S_2(\yb)(v)$ and let
$t_k \downarrow 0$, $(u_k,v_k) \to (u,v)$
with $\yb + t_k v_k \in S_1(\xb + t_k u_k)$.
We get $S_{2}(\yb + t_k v_k) = S_{2}(\yb) + t_k w_k$
 for $w_k = \nabla S_{2}(\yb)v + o(1) \to w$.
 Thus $(\xb,\zb,\yb) + t_k (u_k,w_k,v_k) \in
 \gph \Xi$ and \eqref{eq: TangProdRel*} follows.
\end{proof}

Before we present the sum rule, we propose an auxiliary result,
interesting on its own.
Consider two osc mappings $S_i: \R^n \tto \R^{m_i}$ for $i = 1,2$
and let $\mathcal{P}: \R^n \tto \R^{m_1} \times \R^{m_2}$ be given by
$\mathcal{P}(x) = S_1(x) \times S_2(x)$.
Note that $\mathcal{P} = S_{o} \circ F_{1}$
for $F_{1}: x \to (x,x)$ and
$S_{o}:(q_{1},q_{2}) \rightrightarrows S_{1}(q_{1}) \times S_{2}(q_{2})$
and fix $(\xb,\bar y)=(\xb,(\yb_1,\yb_2)) \in \gph \mathcal{P}$.
\begin{proposition}\label{Pro: ProdMap}
 We have
 \begin{equation*}
  D \mathcal{P}(\xb,\bar y) \subset D S_{1}(\xb,\yb_1) \times D S_{2}(\xb,\yb_2)
 \end{equation*}
 and the opposite inclusion holds if the mapping
 \begin{equation}\label{eq:SubregProd}
     (x,y,q) \tto ((x,x) - (q_1,q_2)) \times (\gph S_{1} - (q_1,y_1))
     \times (\gph S_{2} - (q_2,y_2))
 \end{equation}
 is metrically subregular at $\big((\xb,\yb,(\xb,\xb)),(0,\ldots,0)\big)$
 and we have
 \begin{equation}\label{eq:TCprod}
 T_{\gph S_{1} \times \gph S_{2}}(\xb,\yb_1,\xb,\yb_2) = T_{\gph S_{1}}(\xb,\yb_1) \times T_{\gph S_{2}}(\xb,\yb_2).
\end{equation}
This is true, in particular, if one of the mappings $S_1$, $S_2$ is single-valued
 and differentiable at $\xb$.
 \end{proposition}
\begin{proof}
 Since $x \to (x,x)$ is differentiable and
\begin{equation*}
 \gph S_{0} = \{(q_1,q_2,y_1,y_2) \mv (q_1,y_1) \in \gph S_{1},
 (q_2,y_2) \in \gph S_2\} = \sigma(\gph S_{1} \times \gph S_{2}),
\end{equation*}
where $\sigma$ merely permutes the variables,
Corollary \ref{Cor:S_1_single} readily yields the estimates.

If one of the maps is single-valued and differentiable,
\eqref{eq:TCprod} follows by the same arguments
as used in the proof of Corollary \ref{Cor:S_1_single}.
Moreover, the subregularity of mapping \eqref{eq:SubregProd}
is implied by its metric regularity, which reads
\[
x_1^* \in D^*S_1(\xb,\yb_1)(0), \quad
x_2^* \in D^*S_2(\xb,\yb_2)(0), \quad
x_1^* + x_2^* = 0 \quad \Longrightarrow
\quad x_1^*, x_2^* = 0,
\]
see e.g. \cite[page 737]{BeGfrOut18a}.
This is clearly true if one of the maps is single-valued
and differentiable.
\end{proof}

Mapping $\mathcal{P}$ naturally appears whenever one is dealing
with a mapping given by some binary operation, as we will see next.
In order to state the sum rule, let $S_1,S_2$ be as before
but with $m_1=m_2=:m$ and consider $S=S_{1}+S_{2}$ and
$(\xb,\bar z) \in \gph S$.
Clearly $S$ can be written via $\mathcal{P}$
as $S=F_{2} \circ \mathcal{P}$ for $F_{2}:(y_{1}, y_{2}) \to y_{1}+ y_{2}$.
Thus, the following result follows from
Proposition \ref{Pro: ProdMap} and Corollary \ref{Cor:S_2_single},
where the corresponding ``intermediate'' mapping
$\Xi: \mathbb{R}^{n} \times \mathbb{R}^{m} \rightrightarrows \mathbb{R}^{m} \times \mathbb{R}^{m}$
is given by
\[
\Xi(x,z) := \mathcal{P}(x) \cap F_{2}^{-1}(z)
= \{y=(y_{1},y_{2}) \in \mathbb{R}^{m} \times \mathbb{R}^{m} \mv y_{1} \in S_{1}(x), y_{2} \in S_{2}(x), y_{1} + y_2=z\}.
\]

\begin{theorem}[Sum rule for graphical derivatives] \label{The : SumRuleGder}
If $\Xi$ is inner calm* at $(\xb,\zb)$ wrt $\dom \Xi$ in the fuzzy sense, then
\begin{equation}\label{eq:GDSumRule}
 DS(\xb,\bar z) \subset \bigcup\limits_{ \yb \in \Xi(\bar{x},\bar{z}) } DS_1(\xb,\yb_1) + DS_2(\xb,\yb_2).
\end{equation}
Given $\yb \in \Xi(\bar{x},\bar{z})$ such that
the mapping \eqref{eq:SubregProd} is metrically subregular at $\big((\xb,\yb,(\xb,\xb)),(0,\ldots,0)\big)$
and \eqref{eq:TCprod} holds, we have
\begin{equation}\label{eq:GDSumRuleRev}
 DS_1(\xb,\yb_1) + DS_2(\xb,\yb_2) \subset DS(\xb,\bar z).
\end{equation}
In particular, if $S_1$ is single-valued and differentiable, we get
\[DS(\xb,\bar z) = \nabla S_1(\xb) + DS_2(\xb,\bar z - S_1(\xb)).\]
\end{theorem}

We conclude this section by another application of map $\mathcal{P}$,
namely a product rule,
where $S=S_{1} \cdot S_{2}$ for an osc mapping $S_2:\mathbb{R}^{n} \rightrightarrows \mathbb{R}^{m}$
and a single-valued and differentiable $S_1: \R^n \to (\R^m)^l$.
More precisely, $S_1(x)$ is a matrix of $m$ rows and $l$ columns
and so
\[S(x) = \bigcup_{y \in S_2(x)} S_1(x)^T y =
\bigcup_{y \in S_2(x)} \skalp{y,S_1}(x) = F_{3} \circ \mathcal{P}(x)\]
for $F_{3}: (\R^m)^l \times \R^m \to \R^l$ given by $F(A, y) = A^T y$.
Taking into account the single-valuedness of $S_1$,
instead of $\mathcal{P}(x) \cap F_{3}^{-1}(z)$,
we can use the following simplified ``intermediate'' mapping 
$\widetilde \Xi: \mathbb{R}^{n} \times \mathbb{R}^{l} \rightrightarrows \mathbb{R}^{m}$,
given by
\begin{equation}\label{eq:InterMapPrdRule}
\widetilde \Xi(x,z) := \{y \in S_{2}(x) \mv S_1(x)^T y=z\}. 
\end{equation}

Let $(\xb,\bar z) \in \gph S$ and $(u,w)\in \mathbb{R}^{n} \times \mathbb{R}^{l}$ be a pair of directions.
We also provide estimates for the coderivatives, since
they are new and, more importantly, we will use them in the next section.

\begin{theorem}[Product rule] \label{The : ProdRule}
If $\widetilde\Xi$ is inner calm* at $(\xb,\zb)$ wrt $\dom \widetilde\Xi$ in the fuzzy sense, then
\begin{equation*}
    DS(\xb,\bar z)(u) =
 \bigcup\limits_{\yb \in \widetilde\Xi(\bar{x},\bar{z})}
 \nabla \skalp{\yb,S_1}(\xb)u + S_1(\xb)^T DS_2(\xb,\yb)(u).
\end{equation*}
Moreover, if $\widetilde\Xi$ is inner calm* at $(\xb,\zb)$ in
direction $(u,w)$
wrt $\dom \widetilde\Xi$, then
\begin{equation*}
    D^{*}S((\xb,\bar{z}); (u,w))(z^*) \subset
 \bigcup\limits_{\yb \in \widetilde\Xi(\bar{x},\bar{z})} \
 \bigcup\limits_{v \in D \widetilde\Xi((\bar{x},\bar{z}),\yb)(u,w)}
  \nabla \skalp{\yb,S_1}(\bar{x})^T z^* +
   D^{*}S_2((\bar{x}, \yb);(u,v)) (S_{1}(\bar{x}) z^*)
\end{equation*}
holds for all $z^* \in \R^l$.
Finally, under just the inner semicompactness of $\widetilde\Xi$, we have
\[D^{*}S(\xb,\bar{z})(z^*) \subset
 \bigcup\limits_{\yb \in \widetilde\Xi(\bar{x},\bar{z})} \
  \nabla \skalp{\yb,S_1}(\bar{x})^T z^* +
   D^{*}S_2(\bar{x}, \yb) (S_{1}(\bar{x}) z^*).\]
\end{theorem}
\begin{proof}
 The statement about the graphical derivative again comes from
 Corollary \ref{Cor:S_2_single} and Proposition \ref{Pro: ProdMap},
 taking into account the simple structure of $F_3$.
 
 On the basis of Theorem \ref{The : ConstSetForw},
 mimicking the approach from \cite[Theorem 5.1]{BeGfrOut18a},
 one can derive the chain rule for the directional coderivatives
 based on inner calmness*, which looks like \cite[formula (26)]{BeGfrOut18a},
 but without the additional union (see the difference between
 the estimates in (i) and (ii) of Theorem \ref{The : ConstSetForw}).
 The proof then follows by applying \cite[Corollary 5.1]{BeGfrOut18a}
 and the corresponding inner calm* version of \cite[Corollary 5.2]{BeGfrOut18a}.
\end{proof}
\section{Application: Generalized derivatives of the normal cone mapping}

In this section, we apply the proposed calculus rules to compute
the graphical derivative of the normal cone mapping
and to estimate its directional limiting coderivative.
We will show that our calculus-based approach
is very robust and easy to use
in the case of simple constraints
as well as in the parametrized setting.
We then use the estimates to derive some results regarding
the semismoothness* of the normal cone mapping.

\subsection{Simple constraints}
Consider the simple constraint system $g(x) \in D$,
where $D \subset \R^s$ is a convex polyhedral set and $g: \R^n \to \R^s$
is twice continuously differentiable and denote
\begin{equation}
 \Gamma:=g^{-1}(D)=\{x \in \R^n \mv g(x) \in D\}.
\end{equation}
Moreover, fix $\xb \in \Gamma$ and assume that the constraint map $g(x) - D$
is metrically subregular at $\xb$, which, in turn,
means that the subregularity holds at all $x$ near $\xb$ and thus
\begin{equation*}
 N_{\Gamma}(x)=\nabla g(x)^T N_D(g(x)).
\end{equation*}
This shows that the normal cone mapping $x^\prime \tto  N_{\Gamma}(x^\prime)$
can be written as the product of two maps,
the single-valued $\nabla g$ and
the set-valued $N_D \circ g$.

Since the set-valued part contains a composition,
let us first address this issue.
To this end, consider the mapping $F: \R^n\times\R^s \tto \R^s\times\R^s$
given by
\[F(x^\prime,\lambda^\prime) := (g(x^\prime),\lambda^\prime) - \gph N_D\]
and note that $\gph (N_D \circ g) = F^{-1}(0,0)$.
\begin{lemma}\label{Lem:MapM}
Given $\lambda \in N_D(g(x))$, we have
\[\eta \in D (N_D \circ g)(x,\lambda)(u) \ \Longrightarrow \ 
\eta \in D N_D(g(x),\lambda)(\nabla g(x)u).\]
If $F$ is metrically subregular
at $\big((x,\lambda),(0,0)\big)$ in direction $(u,\eta)$,
the reverse implication holds true as well and
for every $\xi^* \in \R^s$ we have
\[
D^{*}(N_D \circ g)((x,\lambda);(u,\eta))(\xi^*) \ \subset \
\nabla g(x)^T D^{*}N_D((g(x),\lambda);(\nabla g(x) u,\eta))(\xi^*).
\]
\end{lemma}
\begin{proof}
Since $\gph (N_D \circ g) = F^{-1}(0,0)$,
\cite[Proposition 4.1]{GfrOut16} yields that
\[\eta \in D (N_D \circ g)(x,\lambda)(u) \ \Longrightarrow \ 
(0,0) \in DF((x,\lambda),(0,0))(u,\eta)\]
always holds and that the reverse implication, as well as
\[
D^{*}(N_D \circ g)((x,\lambda);(u,\eta))(\xi^*) \ \subset \
\{\xi \mv \exists (a,b): (\xi,-\xi^*) \in D^{*}F(((x,\lambda),(0,0));((u,\eta),0,0))(a,b)\},
\]
hold if $F$ is metrically subregular
 at $\big((x,\lambda),(0,0)\big)$ in direction $(u,\eta)$.
The claims now follow from the estimates of derivatives of $F$.
Indeed, employing the sum rule from Theorem \ref{The : SumRuleGder}
and \cite[Corollary 5.3]{BeGfrOut18a}, respectively, yields
\[(0,0) \in DF((x,\lambda),(0,0))(u,\eta) = (\nabla g(x)u,\eta) - T_{\gph N_D}(g(x),\lambda) \ \iff \
\eta \in D N_D(g(x),\lambda)(\nabla g(x)u)
\]
and
\begin{eqnarray*}
\lefteqn{(\xi,-\xi^*) \in D^{*}F(((x,\lambda),(0,0));((u,\eta),0,0))(a,b)} \\
& \hspace{1.5cm} \Longrightarrow &
\xi \in \nabla g(x)^T D^{*}N_D((g(x),\lambda);(\nabla g(x) u,\eta))(\xi^*). \hspace{4cm}
\end{eqnarray*}
Here, we used the simple formulas for derivatives of
the constant mapping $(x^\prime,\lambda^\prime) \tto \gph N_D$
with the graph $\R^{n+s} \times \gph N_D$,
see \cite[Lemma 6.1]{BeGfrOut18a} for more details.
\end{proof}

Now, we can proceed with the estimates for the normal cone mapping.

\begin{theorem}\label{Pro: AuxResNonPar}
 Given $x^* \in N_{\Gamma}(x)$
 for $x$ near $\xb$, for all $u \in \R^n$ we have
 \begin{eqnarray*}
 DN_{\Gamma}(x,x^*)(u) & \subset &
 \bigcup\limits_{\lambda \in \Lambda(x,x^*)}
 \nabla^2 \skalp{\lambda,g}(x)u
 + \nabla g(x)^T N_{\K_D(g(x),\lambda)}(\nabla g(x)u) \\
 & = &
 \bigcup\limits_{\lambda \in \Lambda(x,x^*)}
 \nabla^2 \skalp{\lambda,g}(x)u + 
 N_{\K_{\Gamma}(x,x^*)}(u),
\end{eqnarray*}
where
\[\Lambda(x,x^*)=\{\lambda \in N_{D}(g(x)) \mv \nabla g(x)^T \lambda = x^*\}.\]
The opposite inclusion holds if $F$ is metrically subregular
 at $\big((x,\lambda),(0,0)\big)$ for every $\lambda \in \Lambda(x,x^*)$ in direction $(u,\eta)$ for every
 $\eta \in N_{\K_D(g(x),\lambda)}(\nabla g(x)u)$.
 This is true, in particular, if the following non-degeneracy condition in direction $u$
 \begin{equation}\label{eq:NonDegApp}
 \nabla g(x)^T \mu = 0, \ \mu \in \Span N_{T_D(g(x))}(\nabla g(x) u)
 \ \Longrightarrow \ \mu = 0
 \end{equation}
 holds. Then, however, for all $w^* \in \R^n$ and $u^* \in DN_{\Gamma}(x,x^*)(u)$,
 we also get
 \begin{eqnarray*}
 \lefteqn{D^{*}N_{\Gamma}((x,x^*); (u,u^*))(w^*)} \\
 & \quad \subset &
 \bigcup\limits_{\lambda \in \Lambda(x,x^*)} \
 \bigcup\limits_{\eta \in D \Lambda((x,x^*),\lambda)(u,u^*)}
  \nabla^2 \skalp{\lambda,g}(x) w^* +
  \nabla g(x)^T D^{*}N_D((g(x),\lambda);(\nabla g(x) u,\eta))(\nabla g(x) w^*).
\end{eqnarray*}
\end{theorem}
\begin{proof}
 Theorem \ref{the:Main_Inner_calmness} (ii) yields that the multiplier mapping
$\Lambda(x,x^*)$
is inner calm* at $(x,x^*)$ wrt its domain in the fuzzy sense.
Noting that $\Lambda$ is precisely the ``intermediate'' mapping 
$\widetilde \Xi$ from \eqref{eq:InterMapPrdRule}
appearing in the product rule from Theorem \ref{The : ProdRule},
we obtain
\begin{equation}\label{eq:aux_estim}
 DN_{\Gamma}(x,x^*)(u) =
 \bigcup\limits_{\lambda \in \Lambda(x,x^*)}
 \nabla^2 \skalp{\lambda,g}(x)u + \nabla g(x)^T D(N_D \circ g)(x,\lambda)(u).
\end{equation}
Next, Lemma \ref{Lem:MapM} gives
\begin{equation}\label{eq:RedLemApp}
    D(N_D \circ g)(x,\lambda)(u) \subset DN_D(g(x),\lambda)(\nabla g(x)u)
    = N_{\K_D(g(x),\lambda)}(\nabla g(x)u),
\end{equation}
where the equality corresponds to an equivalent
reformulation of the reduction lemma \eqref{eq:Reduction_Lemma}.
Moreover, for every $\lambda \in \Lambda(x,x^*)$ it holds that
\[\K_{\Gamma}(x,x^*) = \nabla g(x)^{-1} \K_D(g(x),\lambda).\]
Since $u \tto \nabla g(x)u - \K_D(g(x),\lambda)$ is a polyhedral map
and hence metrically subregular by the Robinson's results \cite[Proposition 1]{Rob81},
we obtain 
\[N_{\K_{\Gamma}(x,x^*)}(u) = \nabla g(x)^T N_{\K_D(g(x),\lambda)}(\nabla g(x)u)\]
and the first claim follows.

The opposite inclusion follows from the opposite
inclusion in \eqref{eq:RedLemApp},
which holds true under the subregularity assumption by Lemma \ref{Lem:MapM}.

The claim that the non-degeneracy condition implies the
subregularity assumption was proven in \cite[Theorem 2.11]{BeGfrOut18}.
Moreover, from Theorem \ref{the:Main_Inner_calmness} (iii)
we get that the non-degeneracy also implies the inner calmness*
of $\Lambda$ at $(x,x^*)$ in direction $(u,u^*)$ wrt its domain.
Theorem \ref{The : ProdRule} thus gives the estimate
\begin{eqnarray*}
 \lefteqn{D^{*}N_{\Gamma}((x,x^*); (u,u^*))(w^*)} \\
 & \hspace{1.5cm} \subset &
 \bigcup\limits_{\lambda \in \Lambda(x,x^*)} \
 \bigcup\limits_{\eta \in D \Lambda((x,x^*),\lambda)(u,u^*)}
  \nabla^2 \skalp{\lambda,g}(x) w^* +
   D^{*}(N_D \circ g)((x, \lambda);(u,\eta)) (\nabla g(x) w^*)
\end{eqnarray*}
and using Lemma \ref{Lem:MapM} again completes the proof. 
\end{proof}

Note that the auxiliary formula
\eqref{eq:aux_estim}
holds without additional assumptions.
Moreover, Theorem \ref{the:Main_Inner_calmness} (i) yields
that the analogous estimate for the standard
limiting coderivative holds for any closed set $D$,
which is Clarke regular near $g(x)$, in particular for any closed convex set $D$.
Finally, note also that the formula for the graphical derivative
reveals that the reduction lemma for polyhedral set $D$,
see \eqref{eq:RedLemApp}, is carried over to set $\Gamma$
with the additional curvature term $\cup_{\lambda \in \Lambda(x,x^*)}
 \nabla^2 \skalp{\lambda,g}(x)u$.

As a specific application we derive the following result
regarding the semismoothness* of $N_{\Gamma}$.
\begin{corollary}\label{Pro:semismoothness}
 If the non-degeneracy condition \eqref{eq:NonDegApp} is satisfied
 for some $x,u \in \R^n$,
 then the normal cone mapping $x^\prime \tto N_{\Gamma}(x^\prime)$ is semismooth* at $(x,x^*)$ in direction $(u,u^*)$ for every $x^*, u^* \in \R^n$,
 i.e.,
 \begin{equation*}
  \skalp{u,w} = \skalp{u^*,w^*} \quad \forall \, (w^*,w) \in \gph D^{*}N_{\Gamma}((x,x^*); (u,u^*)).
 \end{equation*}
\end{corollary}
\begin{proof}
 Clearly, if $x^* \notin N_{\Gamma}(x)$ or $u^* \notin DN_{\Gamma}(x,x^*)(u)$, then
 $D^{*}N_{\Gamma}((x,x^*); (u,u^*))=\emptyset$ and there is nothing to prove. Hence, let $u^* \in DN_{\Gamma}(x,x^*)(u)$
 and consider $w \in D^{*}N_{\Gamma}((x,x^*); (u,u^*))(w^*)$.
 Theorem \ref{Pro: AuxResNonPar} yields the existence
 of $\lambda \in \Lambda(x,x^*)$,
 $\eta \in DN_D(g(x),\lambda)(\nabla g(x)u)$ and
 \[\zeta \in D^{*}N_D((g(x),\lambda);(\nabla g(x) u,\eta))(\nabla g(x) w^*)\]
 such that
 \[u^* = \nabla^2 \skalp{\lambda,g}(x) u + \nabla g(x)^T \eta
 \ \textrm{ and } \
 w = \nabla^2 \skalp{\lambda,g}(x) w^* +
  \nabla g(x)^T \zeta.\]
 Since the polyhedral map $(y,\lambda) \tto N_D(y,\lambda)$ is semismooth* at every point of its graph, see \cite[page 7]{GfrOut19},
 we obtain
 \begin{eqnarray*}
 \skalp{u,w} & = &
 \skalp{u,\nabla^2 \skalp{\lambda,g}(x) w^*} +
 \skalp{u,\nabla g(x)^T \zeta} =
 \skalp{\nabla^2 \skalp{\lambda,g}(x) u, w^*} + 
 \skalp{\nabla g(x) u, \zeta} \\
 & = &
 \skalp{\nabla^2 \skalp{\lambda,g}(x) u, w^*} + 
 \skalp{\eta, \nabla g(x) w^*} =
 \skalp{\nabla^2 \skalp{\lambda,g}(x) u + \nabla g(x)^T \eta,w^*} = \skalp{u^*,w^*}.
 \end{eqnarray*}
\end{proof}

\begin{remark}
 In order to show the semismoothness* of the normal come mapping,
 we need to show the semismoothness* in every direction $(u,u^*)$.
 Clearly, the zero direction $(u,u^*)=(0,0)$ causes no problem.
 In order to deal with direction $(0,u^*) \neq (0,0)$, however,
 we require the non-degeneracy condition to hold for the zero
 direction $u=0$, which is equivalent to the standard non-degeneracy,
 see Remark \ref{Rem:Main}.
 This, in turn, quite significantly simplifies the situation,
 in particular, it implies uniqueness of $\lambda$ and $\eta$,
 see \cite{GfrOut18}.
 This can be avoided if we assume that $u^* \in DN_{\Gamma}(x,x^*)(0)$
 implies $u^* = 0$, which is equivalent with the isolated calmness
 of the normal cone mapping at $(x,x^*)$.
 In such case we only need the non-degeneracy in all non-zero
 directions, which is strictly milder than the standard non-degeneracy
 as demonstrated in Example \ref{Ex}.
\end{remark}

\subsection{Constraints depending on the parameter and the solution}

Given again a convex polyhedral set $D \subset \R^s$
and a twice continuously differentiable function
$g: \R^l \times \R^n \times \R^n \to \R^s$,
consider the feasible set
\begin{equation}
 \Gamma(p,x):=\{z \in \R^n \mv g(p,x,z) \in D\}
\end{equation}
depending on the parameter $p$ as well as on the decision
variable $x$.
Here, given a reference point $(\pb,\xb)$ with $g(\pb,\xb,\xb) \in D$,
we require the existence of $\kappa>0$ such that for all $(p,x,z)$ belonging to a neighbourhood of $(\pb,\xb,\xb)$ the metric inequality
  \begin{equation}\label{eq : IneqMain}
   \dist{z,\Gamma(p,x)} \leq \kappa \dist{g(p,x,z),D}
  \end{equation}
holds. This yields, in particular,
that for all $(p,x,z)\in\gph\Gamma$ sufficiently close to $(\pb,\xb,\xb)$ the mapping $g(p,x,\cdot)-D$ is metrically subregular at $(z,0)$ with modulus $\kappa$. Hence, the normal cone mapping
$(p,x) \tto N_{\Gamma(p,x)}(x) =:\mathcal{N}^{\Gamma}(p,x)$
can be again written as the product
\[\mathcal{N}^{\Gamma}(y) = \beta(y)^T N_D(\tilde g(y))\]
for $\tilde g(p,x)=g(p,x,x)$, $\beta(p,x)=\nabla_3 g(p,x,x)$ and $y=(p,x)$.
Given $x^* \in \mathcal{N}^{\Gamma}(y)$, we denote
\[\Lambda(y,x^*) =
\{\lambda \in N_D(\tilde g(y)) \mv \beta(y)^T \lambda = x^*\}.
\]

Since the modulus $\kappa$ of subregularity of $g(p,x,\cdot)-D$
does not depend of $p$ and $x$, we infer the existence of
$\lambda \in \Lambda(y,x^*)$ with $\norm{\lambda} \leq \kappa \norm{x^*}$
and the proof of Corollary \ref{cor:inner_semicompactness}
yields inner semicompactness of $\Lambda$ at $(y,x^*)$ wrt its domain. 
Moreover, Theorem \ref{the:Main_Inner_calmness} and
Remark \ref{Rem:Main} (1.) then imply even inner calmness*
of $\Lambda$ in the fuzzy sense.
The same arguments as in the previous case
thus yield the following estimates.
\begin{theorem}
For every $v \in\R^l\times\R^n$ we have
\begin{equation*}
 D\mathcal{N}^{\Gamma}(y,x^*)(v) \subset \bigcup\limits_{\lambda \in \Lambda(y,x^*)}
 \nabla \skalp{\lambda,\beta}(y)v + \beta(y)^T N_{\K_{D}(\tilde g(y),\lambda)}(\nabla \tilde g(y) v). 
\end{equation*}
On the other hand, given $v \in\R^l\times\R^n$, $\lambda \in \Lambda(y,x^*)$
  and $\eta \in N_{\K_{D}(\tilde  g(y),\lambda)}(\nabla \tilde  g(y)(v))$,
  we have
  \begin{equation*} \label{EqInclDG1}
     \nabla \skalp{\lambda,\beta}(y)v + \beta(y)^T\eta \in D\mathcal{N}^{\Gamma}(y,x^*)(v),
  \end{equation*}
  provided the mapping
  $(y^\prime,\lambda^\prime) \tto
  \big(\tilde g(y^\prime), \lambda^\prime \big)-\gph N_D$
  is metrically subregular at $\big((y,\lambda),(0,0)\big)$ in direction $(v,\eta)$, in particular, if
  \begin{equation}\label{eq:NonDegAppParamSol_ms}
  \nabla \tilde g(y)^T \mu = 0, \ \mu \in \Span N_{T_D(\tilde g(y))}(\nabla \tilde g(y) v)
 \ \Longrightarrow \ \mu = 0.
 \end{equation}
 Moreover, if also
 \begin{equation}\label{eq:NonDegAppParamSol}
 \beta(y)^T \mu = 0, \ \mu \in \Span N_{T_D(\tilde g(y))}(\nabla \tilde g(y) v)
 \ \Longrightarrow \ \mu = 0
 \end{equation}
 holds, then for $u^* \in D\mathcal{N}^{\Gamma}(y,x^*)(v)$
  and arbitrary $w^* \in \R^n$ we get the estimate
 \begin{eqnarray*}
  \lefteqn{D^{*}\mathcal{N}^{\Gamma}((y,x^*); (v,u^*))(w^*)} \\
  & \subset &
  \bigcup\limits_{\lambda \in \Lambda(y,x^*)} \
 \bigcup\limits_{\eta \in D \Lambda((y,x^*),\lambda)(v,u^*)}
  \nabla \skalp{\lambda,\beta}(y)^T w^* + \nabla \tilde g(y)^T
   D^{*}N_D((\tilde g(y), \lambda);(\nabla \tilde g(y)v,\eta)) (\beta(y) w^*).
 \end{eqnarray*}
\end{theorem}

We point out that here we need to impose the
two non-degeneracy conditions \eqref{eq:NonDegAppParamSol}
and \eqref{eq:NonDegAppParamSol_ms}.
In general, these two conditions are mutually incomparable,
as can be seen from
\[\nabla \tilde g(p,x) = (\nabla_1 g(p,x,x), \nabla_2 g(p,x,x) + \beta(p,x)).\]
If, however, $\nabla_2 g = 0$, in particular if
the feasible set $\Gamma$ depends only of $p$
but not on $x$, we get $\nabla \tilde g(p,x) = (\nabla_1 g(p,x,x), \beta(p,x))$
and, in this case, \eqref{eq:NonDegAppParamSol} implies
\eqref{eq:NonDegAppParamSol_ms}.
Naturally, these conditions become equivalent if also $\nabla_1 g = 0$,
which covers the previous case of non-parametrized constraints.

Let us add a few comments on existing results on the topic.
This model was investigated already in \cite{MoOut07} by using the
classical calculus of Mordukhovich.
Recently, in \cite{GfrOut18}, Gfrerer and Outrata also computed the derivatives
of the normal cone mapping $(p^\prime,x^\prime) \tto N_{\Gamma(p^\prime,x^\prime)}(x^\prime)$.
Instead of \eqref{eq : IneqMain} or \eqref{eq:NonDegAppParamSol},
however, they imposed the standard non-degeneracy condition \eqref{eq:NonDegAppParamSol} for $v=0$,
which guarantees uniqueness of the multipliers $\lambda$ and $\eta$.
The reason was that they were in fact relying on inner calmness
when using calculus rules, see, e.g., \cite[Lemmas 4.1 and 4.2]{GfrOut18}.

On the other hand, in \cite{BeGfrOut18}, we worked with the same
assumption \eqref{eq : IneqMain} and proved the same result
for the graphical derivative. The explicit estimation of the coderivative
was, however, bypassed by addressing the related stability issues
directly, see \cite[Theorem 6.1]{BeGfrOut18}.

Let us also mention the paper \cite{GfrMo17}
dealing with the constraints depending on $p$, but not on $x$.
This paper delivers stronger results
even in a more general setting that goes beyond polyhedrality.
In particular, the formula for the graphical derivative
is in fact valid without the additional subregularity assumption,
see \cite[Theorem 3.3]{GfrMo17}.
Naturally, the same applies to the non-parametrized setting.

We point out that in \cite{BeGfrOut18,GfrOut18} the formula
for the graphical derivative is derived from \cite[Theorem 5.3]{GfrMo17}.
Our calculus-based approach seems to be a bit simpler
and it is easily applicable regardless of whether
$\Gamma$ is a fixed set or it depends on $p$ and $x$.

We conclude this section by the corresponding semismoothness* result.
\begin{corollary}
 Under \eqref{eq:NonDegAppParamSol_ms} and \eqref{eq:NonDegAppParamSol},
 the normal cone mapping
 $(p^\prime,x^\prime) \tto N_{\Gamma(p^\prime,x^\prime)}(x^\prime)$
 is semismooth* at $(y,x^*)=(p,x,x^*)$ in direction $(v,u^*)=(q,u,u^*)$
 for every $x^*, u^* \in \R^n$.
\end{corollary}
\begin{proof}
 The proof follows by the same arguments as the proof of
 Corollary \ref{Pro:semismoothness}.
 We only briefly clarify that the difference between $\nabla \tilde g$
 and $\beta$ causes no problem.
 Indeed, looking at the above theorem and the proof of Corollary \ref{Pro:semismoothness},
 we now get
 \[\zeta \in D^{*}N_D((\tilde g(y),\lambda);(\nabla \tilde g(y) v,\eta))(\beta(y) w^*)\]
 such that
 \[u^* = \nabla \skalp{\lambda,\beta}(y) v + \beta(y)^T \eta
 \ \textrm{ and } \
 w = \nabla \skalp{\lambda,\beta}(y)^T w^* + \nabla \tilde g(y)^T \zeta.\]
 The computation then goes as follows
 \begin{eqnarray*}
 \skalp{v,w} & = &
 \skalp{v,\nabla \skalp{\lambda,\beta}(y)^T w^*} +
 \skalp{v,\nabla \tilde g(y)^T \zeta} =
 \skalp{\nabla \skalp{\lambda,\beta}(y) v, w^*} + 
 \skalp{\nabla \tilde g(y) v, \zeta} \\
 & = &
 \skalp{\nabla \skalp{\lambda,\beta}(y) v, w^*} + 
 \skalp{\eta, \beta(y) w^*} =
 \skalp{\nabla \skalp{\lambda,\beta}(y) v + \beta(y)^T \eta,w^*} = \skalp{u^*,w^*}.
 \end{eqnarray*}
\end{proof}

\section*{Final Remarks}

Motivated by the recent success in the computation of the graphical derivative
of the normal cone mapping to sets with constraint structure,
in this paper, we identified the role played by (fuzzy) inner calmness*
and the underlying calculus principles.
We hope that it provides some helpful insights into the matter.
We also hope that the new notions will turn out useful also in other circumstances.
We believe that some reasonable topics for further study are readily available:
Can isolated calmness be effectively used as a sufficient condition for inner calmness?
Is fuzzy inner calmness* also necessary for validity of the estimate \eqref{eq:MainIssue}?
Can the new formulas for tangent cones yield useful estimates for regular normal cones
by polarization and, in turn, enrich the theory of stationarity conditions?
We plan to address these issues in the forthcoming paper.



\end{document}